\documentclass[a4paper, 10pt, notitlepage]{article}

\usepackage{amsthm, amsmath, amssymb, latexsym}
\usepackage{mathrsfs}
\usepackage[multiple]{footmisc}

\theoremstyle{plain}
\newtheorem{theorem}{Theorem}
\newtheorem{lemma}{Lemma}
\newtheorem{corollary}{Corollary}
\newtheorem{proposition}{Proposition}

\theoremstyle{definition}
\newtheorem{definition}{Definition}
\newtheorem{example}{Example}

\theoremstyle{remark}
\newtheorem{remark}{Remark}

\DeclareMathOperator{\co}{co}
\DeclareMathOperator{\cl}{cl}
\DeclareMathOperator{\extr}{ext}
\DeclareMathOperator{\dist}{dist}
\DeclareMathOperator{\dom}{dom}
\DeclareMathOperator{\interior}{int}
\DeclareMathOperator*{\argmin}{arg\,min}
\DeclareMathOperator*{\argmax}{arg\,max}

\author{Dolgopolik M.V.\footnote{Saint Petersburg State University, Saint Petersburg, Russia}
\footnote{Institute of Problems of Mechanical Engineering, Saint Petersburg, Russia}\footnote{The author was supported
by the Russian Foundation for Basic Research under grant no.~16--31-00056.}}
\title{A convergence analysis of the method of codifferential descent}

\begin{document}

\maketitle

\begin{abstract}
This paper is devoted to a detailed convergence analysis of the method of codifferential descent (MCD) developed by
professor V.F. Demyanov for solving a large class of nonsmooth nonconvex optimization problems. We propose a
generalization of the MCD that is more suitable for applications than the original method, and that utilizes only a part
of a codifferential on every iteration, which allows one to reduce the overall complexity of the method. With the use of
some general results on uniformly codifferentiable functions obtained in this paper, we prove the global convergence of
the generalized MCD in the infinite dimensional case. Also, we propose and analyse a quadratic regularization of the
MCD, which is the first general method for minimizing a codifferentiable function over a convex set. Apart from
convergence analysis, we also discuss the robustness of the MCD with respect to computational errors, possible step size
rules, and a choice of parameters of the algorithm. In the end of the paper we estimate the rate of convergence
of the MCD for a class of nonsmooth nonconvex functions that arise, in particular, in cluster analysis. We prove
that under some general assumptions the method converges with linear rate, and it convergence quadratically, provided a
certain first order sufficient optimality condition holds true.
\end{abstract}

\section{Introduction}

The class of codifferentiable functions was introduced by professor V.F. Demyanov
\cite{Demyanov_InCollection_1988,Demyanov1988,Demyanov1989} in the late 1980s. The main feature of this class of
nonsmooth functions is the fact that the approximation provided by the codifferential is \textit{continuous} with
respect to the reference point and nonhomogeneous w.r.t. the direction, while other standard tools of
nonsmooth analysis, such as various generalized derivatives and subdifferentials, provide approximations that
are homogeneous w.r.t. the direction, but are continuous w.r.t. the reference point only in the smooth case.
Furthermore, the class of codifferentiable functions includes all d.c. functions, forms a vector lattice closed under
all standard algebraic operations and composition, and admits a simple and \textit{exact} calculus \cite{DemRub_book}
that can be easily implemented on a computer (cf. \textit{fuzzy} or \textit{approximate} calculus rules in
\cite{Morukhovich,Penot_book}). The codifferential calculus was extended to the infinite dimensional case 
in \cite{Zaffaroni,Dolgopolik,Dolgopolik_AbstractConv}.

Codifferentiable functions were introduced as a natural modification of the class of quasidifferentiable functions.
Recall that a function $f$ is called quasidifferentiable at a point $x$, if $f$ is directionally differentiable at $x$,
and there exists a pair of convex compact sets $\mathcal{D} f(x) = [\underline{\partial} f(x), \overline{\partial}
f(x)]$, called \textit{a quasidifferential} of $f$ at $x$, such that 
$f'(x, h) = \max_{v \in \underline{\partial} f(x)} \langle v, h \rangle 
+ \min_{w \in \overline{\partial} f(x)} \langle w, h \rangle$. Despite their many useful properties and various
applications \cite{DemRub_book,QuasidiffMechanics,Quasidifferentiability_book,PalRecUrb,Uderzo}, it appeared to be very
difficult to design effective numerical methods for minimizing general quasidifferentiable functions, and only
methods for minimizing some specific classes of these functions were proposed (see, e.g.,
\cite{Barirog_Survey,Bagirov2002,LudererWeigelt2003}). Since quasidifferentiable functions are locally Lipschitz
continuous, one can use general methods of nonsmooth optimization, such as 
bundle methods \cite{Makela2002,HaaralaMiettinenMakela,HareSagatizabal,FuduliGaudioso}, 
gradient sampling methods \cite{BurkeLewisOverton,Kiwiel2010,HareNutini,CurtisQue2013}, 
quasi-Newton methods \cite{LewisOverton2013,KeskarWachter},
discrete gradient methods \cite{BarigovKarasozenSezer,BagirovRubinovZhang} and many others
\cite{ConnScheinbergVicent_book,KarmitsaBagriovMakela2012,BagirovKarmitsaMakela_book}, in order to minimize a
quasidifferentiable function. However, these methods do not utilize any additional information about function's
behaviour that a quasidifferential can provide.

Unlike the case of quasidifferentiable functions, it turned out to be possible to design an effective numerical method
for minimizing codifferentiable functions called \textit{the method of codifferential descent} (MCD)
\cite{DemRub_book}. A modification of the MCD aiming at reducing the complexity of the original method was proposed in
\cite{DemBagRub}. Methods for minimizing nonsmooth convex and d.c. functions combining the ideas of the MCD and bundle
methods were developed in \cite{BagGanUgonTor,TorKarBag,BagUgon,TorBagKar}. A trust region method for minimizing
codifferentiable functions was studied in \cite{Andramonov}. Finally, an infinite-dimensional version of the MCD was
applied to some problems of the calculus of variations \cite{DemyanovTamasyan2011,DemyanovTamasyan2014} and optimal
control problems \cite{Fominyh}. However, theoretical results on a convergence of the MCD are very scarce, and only
some results on the global convergence of this method and its modifications are known in the finite dimensional case.
Furthermore, in \cite{DemRub_book,DemBagRub} the global convergence of the MCD was proved under the assumption that the
objective function is codifferentiable uniformly in directions in a neighbourhood of a limit point, and it is unclear
how to verify this assumption in any particular case, while the convergence analysis in
\cite{BagGanUgonTor,TorKarBag,BagUgon,TorBagKar} heavily relies on the convexity or the d.c. structure of the objective
function. Let us finally note that the MCD, as well as the method of truncated codifferential \cite{DemBagRub}, are
unrealisable in the general case, since they cannot be directly implemented in practice 
(see Remark~\ref{Remark_MCD_Unrealisable} below).

The main goal of this paper is to provide a comprehensive convergence analysis of the MCD in the general case. To this
end, we introduce and study a new class of nonsmooth functions called \textit{uniformly codifferentiable} functions. We
prove that this class of functions contains almost all codifferentiable functions appearing in applications. We also
introduce a new generalized version of the MCD that, in essence, coincides with the practical implementation of the
original MCD, and is reduced to the MCD or the method of truncated codifferential \cite{DemBagRub} under a certain
choice of parameters. Furthermore, our version of the MCD allows one to use only a part of a codifferential on every
iteration, which might significantly reduce the overall complexity of the method for some specific problems. Under some
natural and easily verifiable assumptions we prove the global convergence of the generalized MCD in the infinite
dimensional case, and for the first time prove the convergence of a non-stationarity measure to zero for a
sequence generated by the MCD. Finally, we introduce and study a \textit{quadratic regularization of the MCD}, which is
the first general method for minimizing codifferentiable functions over convex sets. This method can be viewed as a
generalization of the Pschenichnyi-Pironneau-Polak (PPP) algorithm for min-max problems
\cite{Pschenichny,PironneauPolak,DaugavenMalozemov,Polak} to the case of codifferentiable functions.

In the end of the paper we for the first time study the rate of convergence of the MCD. Since an analysis of the rate of
convergence of the MCD in the general case requires the use of very complicated and cumbersome second order
approximations of codifferentiable functions such as the so-called second order codifferentials \cite{DemRub_book}, for
the sake of simplicity we study only the rate of convergence of the quadratic regularization of the MCD for a particular
class of nonsmooth nonconvex functions that arise, in particular, in cluster analysis \cite{DemBagRub}. Under
some general assumptions we prove that the quadratic regularization of the MCD converges with linear rate, and it
converges quadratically, if the limit point of a sequence generated by the method satisfies certain first order
sufficient optimality conditions. These results can be viewed as an interesting example that sheds some light on the
way the MCD performs for various problems, and how the rate of convergence of this method can be estimated in the
general case.

The paper is organized as follows. In Section~\ref{Section_MCD} we present a general scheme of the MCD and prove the
global convergence of the method. In this section we also analyse the robustness of the MCD with respect to
computational errors, and discuss possible step size rules and a choice of parameters of the method.
Section~\ref{Section_QR_MCD} is devoted to the quadratic regularization of the MCD, while the rate of convergence of
this method for a class of nonsmooth nonconvex functions is studied in Section~\ref{Section_Rate_of_Convergence}. For
the reader's convenience, some general results on codifferentiable function are given in
Section~\ref{Section_CodifferentiableFunctions}. This section can be viewed as a brief, but thorough and almost
self-contained introduction into the codifferential calculus. However, it should be noted that many results presented in
Section~\ref{Section_CodifferentiableFunctions}, such as the Lipschitz continuity of codifferentiable functions and
general results on uniformly codifferentiable functions, are completely new.

\section{Codifferentiable functions}
\label{Section_CodifferentiableFunctions}

Let $\mathcal{H}$ be a real Hilbert space, and let a function $f \colon U \to \mathbb{R}$ be defined in a neighbourhood
$U$ of a point $x \in \mathcal{H}$. Recall that $f$ is said to be \textit{codifferentiable} at $x$, if there exist
weakly compact convex sets $\underline{d} f(x), \overline{d} f(x) \subset \mathbb{R} \times \mathcal{H}$ such that for
any $\Delta x \in \mathcal{H}$ one has
\begin{equation*}
\begin{split}
  \lim_{\alpha \to +0} \frac{1}{\alpha} \Big| f(x + \alpha \Delta x) - f(x)
  &- \max_{(a, v) \in \underline{d} f(x)} \big( a + \alpha \langle v, \Delta x \rangle \big) \\
  &- \min_{(b, w) \in \overline{d} f(x)} \big( b + \alpha \langle w, \Delta x \rangle \big) \Big| = 0,
\end{split}
\end{equation*}
and
\begin{equation} \label{CodiffApproxZeroAtZero}
  \max_{(a, v) \in \underline{d} f(x)} a + \min_{(b, w) \in \overline{d} f(x)} b = 0.
\end{equation}
Here $\langle \cdot, \cdot \rangle$ is the inner product in $\mathcal{H}$. The pair 
$D f(x) = [ \underline{d} f(x), \overline{d} f(x) ]$ is called \textit{a codifferential} of $f$ at $x$, the set
$\underline{d} f(x)$ is called \textit{a hypodifferential} of $f$ at $x$, while the set $\overline{d} f(x)$ is referred
to as \textit{a hyperdifferential} of $f$ at $x$.

\begin{remark}
One can verify that the function $f$ is codifferentiable at $x$ iff there exist continuous convex functions
$\Phi, \Psi \colon \mathcal{H} \to \mathbb{R}$ such that $\Phi(0) - \Psi(0) = 0$, and for any $\Delta x \in \mathcal{H}$
one has
$$
  \lim_{\alpha \to +0} \frac{1}{\alpha} \Big| f(x + \alpha \Delta x) - f(x) 
  - \big( \Phi(\alpha \Delta x) - \Psi(\alpha \Delta x) \big) \Big| = 0
$$
(see~\cite{Dolgopolik_AbstractConv}, Example~3.10). Thus, the codifferentiable functions are exactly those nonsmooth
functions that can be locally approximated by a d.c. function, which, in particular, implies that any d.c. function is
codifferentiable.
\end{remark}

Observe that a codifferential of $f$ at $x$ is not unique. Indeed, if $D f(x)$ is a codifferential of $f$ at $x$,
and $C \subset \mathbb{R} \times \mathcal{H}$ is a weakly compact convex set, then the pair
$[\underline{d} f(x) + C, \overline{d} f(x) - C]$ is a codifferential of $f$ at $x$, as well. Note also that not
all codifferentials of $f$ at $x$ have the form $[\underline{d} f(x) + C, \overline{d} f(x) + C]$ for some weakly
compact convex set $C$. For instance, applying the inequality $f(x) - f(y) \ge \langle \nabla f(y), x - y \rangle$ one
can check that for any $c \ge 0$ the pair $[\co\{ (-y^2, 2y) \colon y \in [-c, c] \}, \{ 0 \}]$ is a codifferential
of the function $f(x) = x^2$ at the point $x = 0$.

If there exists a codifferential of $f$ at $x$ of the form $[\underline{d} f(x), \{ 0 \}]$, then the function $f$ is
called \textit{hypodifferentiable} at $x$. Similarly, if there exists a codifferential of $f$ at $x$ of the form
$[\{ 0 \}, \overline{d} f(x)]$, then $f$ is said to be \textit{hyperdifferentiable} at $x$. The function $f$ is called
\textit{continuously codifferentiable} at the point $x$, if $f$ is codifferentiable in a neighbourhood of $x$ (i.e. at
every point of this neighbourhood), and there exists a \textit{codifferential mapping} $D f(\cdot)$ defined in
neighbourhood of $x$ and such that the multifunctions $\underline{d} f(\cdot)$ and $\overline{d} f(\cdot)$ are Hausdorff
continuous at $x$. In this case the codifferential mapping $D f(\cdot)$ is called \textit{continuous} at $x$.
Finally, a function $f$ defined in a neighbourhood $U$ of a set $A \subset \mathcal{H}$ is called continuously
codifferentiable on the set $A$, if $f$ is codifferentiable at every point of the set $A$, and there exists a
codifferential mapping $D f(\cdot)$ such that the corresponding multifunctions $\underline{d} f(\cdot)$ and
$\overline{d} f(\cdot)$ are Hausdorff continuous on the set $A$. Continuously hypodifferentiable and continuously
hyperdifferentiable functions are defined in the same way.

\begin{remark}
From this point onwards, when we consider a continuously codifferentiable function, we suppose that the corresponding
codifferential mapping $D f(\cdot)$ is continuous (i.e. we do not consider discontinuous codifferential mappings for
continuously codifferentiable functions).
\end{remark}

Let us give some simple examples of continuously codifferentiable functions.

\begin{example} \label{Example_Differentiable}
Let $f$ be G\^{a}teaux differentiable at a point $x \in \mathcal{H}$. Then $f$ is codifferentiable at $x$, and both
pairs $[\{(0, \nabla f(x))\}, \{0\}]$ and $[\{0\}, \{(0, \nabla f(x)\}]$ are codifferentials of $f$ at $x$, i.e. $f$ is
both hypodifferentiable and hyperdifferentiable at $x$. If, in addition, $f$ is continuously differentiable at $x$, then
$f$ is continuously codifferentiable at this point.
\end{example}

\begin{example}
Let $f(x) = \| A x + b \|$, where $A \colon \mathcal{H} \to \mathcal{H}$ is a bounded linear operator, and 
$b \in \mathcal{H}$ is fixed. Then
\begin{align*}
  f(x + \Delta x) - f(x) &= \max_{v \in B(0, 1)} \big\langle v, A(x + \Delta x) + b \big\rangle - \| A x + b \| \\
  &= \max_{(a, v) \in \underline{d} f(x)} \big( a + \langle v, \Delta x \rangle \big),
\end{align*}
where 
$$
  \underline{d} f(x) = 
  \Big\{ \Big(\langle v, A x + b \rangle - \| A x + b \|, A^* v \Big) \Bigm| v \in B(0, 1) \Big\},
$$
the operator $A^*$ is the adjoint of $A$, and $B(y, r) = \{ z \in \mathcal{H} \mid \| z - y \| \le r \}$.
The set $\underline{d} f(x)$ is convex and weakly compact (the weak compactness follows from the fact that the
operator $A^*$ is an endomorphism of the space $\mathcal{H}$ endowed with the weak topology). Furthermore, it is easy to
see that the mapping $\underline{d} f(\cdot)$ is Hausdorff continuous. Thus, the function $f(x) = \| A x + b \|$ is
continuously hypodifferentiable on the entire space $\mathcal{H}$.
\end{example}

\begin{example}
Let $f$ be a piecewise affine function of the form $f(x) = f_1(x) + f_2(x)$ with
$$
  f_1(x) = \max_{i \in I} \big( a_i + \langle v_i, x \rangle \big), \quad
  f_2(x) = \min_{j \in J} \big( b_j + \langle w_j, x \rangle \big),
$$
where $I$ and $J$ are some finite index sets, $a_i, b_j \in \mathbb{R}$ and $v_i, w_j \in \mathcal{H}$. Then the
function $f$ is continuously codifferentiable on $\mathcal{H}$, and one can define
$\underline{d} f(x) = \{ (a_i + \langle v_i, x \rangle - f_1(x), v_i) \mid i \in I \}$ and
$\overline{d} f(x) = \{ (b_j + \langle w_j, x \rangle - f_2(x), w_j) \mid j \in J \}$.
\end{example}

\begin{example} \label{Example_ConvexFunction}
Let $f$ be a convex function, and $U \subset \mathcal{H}$ be a bounded open set such that $f$ is Lipschitz continuous
on $U$. By the definition of subgradient for any $y \in U$, $v \in \partial f(y)$ and $x \in \mathcal{H}$ one has 
$f(x) \ge f(y) + \langle v, x - y \rangle$, and this inequality turns into an equality when $x = y$. Therefore, 
$f(x) = \max_{(a, v) \in C} (a + \langle v, x \rangle)$ for any $x \in U$, where
$$
  C = \big\{ (f(y) - \langle v, y \rangle, v) \in \mathbb{R} \times \mathcal{H} \bigm| 
  v \in \partial f(y), \: y \in U \big\}.
$$
Hence taking into account the fact that $U$ is an open set one obtains that for any $x \in U$ there exists $r > 0$ such
that
$$
  f(x + \Delta x) - f(x) = \max_{(a, v) \in \underline{d} f(x)} (a + \langle v, \Delta x \rangle)
  \quad \forall \Delta x \in B(x, r),
$$
where
$$
  \underline{d} f(x) = \cl \co \Big\{ 
  \big( f(y) - f(x) + \langle v, x - y \rangle, v \big) \in \mathbb{R} \times \mathcal{H} \Bigm|
  v \in \partial f(y), \: y \in U \Big\}.
$$
Note that the set $\underline{d} f(x)$ is bounded (and thus weakly compact) due to the fact that $f$ is Lipschitz
continuous on $U$. Furthermore, one can verify that the mapping $\underline{d} f(\cdot)$ is Hausdorff continuous.
Hence, in particular, one gets that a continuous convex function $f \colon \mathcal{H} \to \mathbb{R}$ is continuously
hypodifferentiable at every point $x \in \mathcal{H}$. On the other hand, if $f$ is continuously codifferentiable at a
given point, then it is Lipschitz continuous in a neighbourhood of this point (see
Corollary~\ref{Corollary_LipschitzContinuity} below). Thus, a proper convex function 
$f \colon \mathcal{H} \to \mathbb{R} \cup \{ \pm \infty \}$ is continuously hypodifferentiable at a point $x$ iff
$x \in \interior \dom f$, and $f$ is continuous at this point. If $\mathcal{H}$ is finite dimensional, then $f$ is
continuously hypodifferentiable at every interior point of its effective domain. 
\end{example}

Let a function $f$ be defined and codifferentiable on an open set $U \subset \mathcal{H}$. Observe that without loss of
generality one can suppose that
\begin{equation} \label{BothApproxZeroAtZero}
  \max_{(a, v) \in \underline{d} f(x)} a = \min_{(b, w) \in \overline{d} f(x)} b = 0 \quad \forall x \in U,
\end{equation}
since otherwise one can use the codifferential mapping 
$\widehat{D f}(\cdot) = [\widehat{\underline{d} f}(\cdot), \widehat{\overline{d} f}(\cdot)]$ of the function $f$ of 
the form
\begin{gather*}
  \widehat{\underline{d} f}(x) = \big\{ (a - a(x), v) \bigm| (a, v) \in \underline{d} f(x) \big\}, \\
  \widehat{\overline{d} f}(x) = \big\{ (b - b(x), w) \bigm| (b, w) \in \overline{d} f(x) \big\}
\end{gather*}
where $a(x) = \max_{(a, v) \in \underline{d} f(x)} a$, and $b(x) = \min_{(b, w) \in \overline{d} f(x)} b$ 
(recall that $a(x) = - b(x)$ by \eqref{CodiffApproxZeroAtZero}). Observe also that if the codifferential mapping 
$D f(\cdot)$ is continuous, then the codifferential mapping $\widehat{D f}(\cdot)$ is continuous as well. Note also
that the codifferential mappings from the examples above satisfy equalities \eqref{BothApproxZeroAtZero}. Therefore,
hereinafter we suppose that \eqref{BothApproxZeroAtZero} always holds true.

Let a function $f$ be defined in a neighbourhood of a point $x \in \mathcal{H}$ and codifferentiable at this point.
Then, as it is easy to see, the function $f$ is directionally differentiable at $x$, and 
$f'(x, h) = \Phi'(0, h) + \Psi'(0, h)$ for all $h \in \mathcal{H}$, where
$$
  \Phi(y) = \max_{(a, v) \in \underline{d} f(x)} \big( a + \langle v, y \rangle \big), \quad
  \Psi(y) = \min_{(b, w) \in \overline{d} f(x)} \big( b + \langle w, y \rangle \big).
$$
Applying the theorem about the subdifferential of the supremum of a family of convex functions (see, e.g.,
\cite[Thm.~4.2.3]{IoffeTihomirov}) and equalities~\eqref{BothApproxZeroAtZero} one obtains that $f$ is
quasidifferentiable at $x$, and the pair $\mathscr{D}f(x) = [\underline{\partial} f(x), \overline{\partial} f(x)]$ with
\begin{equation} \label{Quasidiff_Of_Codiff_Func}
  \underline{\partial} f(x) = \big\{ v \in \mathcal{H} \bigm| (0, v) \in \underline{d} f(x) \big\}, \quad
  \overline{\partial} f(x) = \big\{ w \in \mathcal{H} \bigm| (0, w) \in \overline{d} f(x) \big\}
\end{equation}
is a quasidifferential of $f$ at $x$. Note that both sets $\underline{\partial} f(x)$ and
$\overline{\partial} f(x)$ are nonempty, weakly compact and convex. Hereinafter, we consider only the
quasidifferential mapping $\mathscr{D} f(\cdot)$ of the form \eqref{Quasidiff_Of_Codiff_Func} corresponding to a chosen
co\-dif\-ferential mapping of the function $f$.

With the use of the standard first order optimality condition in terms of directional derivative we arrive at the
following result, which is well-known in the finite dimensional case.

\begin{proposition} \label{Prp_NessOptCond}
Let a function $f$ be defined in a neighbourhood of a point $x^* \in \mathcal{H}$ and codifferentiable at this point.
Suppose also that $x^*$ is a point of local minimum of the function $f$. Then
\begin{equation} \label{CodiffOptCond}
  0 \in \underline{d} f(x^*) + (0, w) \quad \forall (0, w) \in \overline{d} f(x^*).
\end{equation}
Moreover, optimality condition \eqref{CodiffOptCond} holds true iff $f'(x^*, \cdot) \ge 0$; thus, it is independent of
the choice of a codifferential, i.e. if \eqref{CodiffOptCond} holds true for one codifferential of $f$ at $x^*$, then it
holds true for all codifferentials of $f$ at $x^*$.
\end{proposition}

\begin{proof}
From the fact that $x^*$ is a point of local minimum of $f$ it follows that $f'(x^*, \cdot) \ge 0$. Hence by the
definition of quasidifferentiable function one has
$$
  f'(x^*, h) = \max_{v \in \underline{\partial} f(x^*)} \langle v, h \rangle
  + \min_{w \in \overline{\partial} f(x^*)} \langle w, h \rangle \ge 0 \quad \forall h \in \mathcal{H}.
$$
Clearly, this inequality holds true iff
$$
  \max_{v \in \underline{\partial} f(x^*)} \langle v + w, h \rangle \ge 0 \quad 
  \forall h \in \mathcal{H} \quad \forall w \in \overline{\partial} f(x^*).
$$
In turn, this condition is valid iff $0 \in \underline{\partial} f(x^*) + w$ for all 
$w \in \overline{\partial} f(x^*)$ or, equivalently, iff \eqref{CodiffOptCond} holds true
(see \eqref{Quasidiff_Of_Codiff_Func}). Finally, note that since the condition $f'(x^*, \cdot) \ge 0$ is independent of
the choice of a codifferential, the optimality condition \eqref{CodiffOptCond} is independent of the choice of a
codifferential as well.	 
\end{proof}

\begin{remark} 
{(i)~The independence of optimality conditions in quasidifferential programming of the choice of quasidifferentials was
studied in \cite{Luderer,LudererRosigerWurker}.
}

\noindent{(ii)~Note that from equalities \eqref{BothApproxZeroAtZero} it follows that if $f$ is codifferentiable at a
point $x$, then there exists at least one pair $(0, w) \in \overline{d} f(x)$, i.e. the optimality condition
\eqref{CodiffOptCond} cannot be satisfied vacuously. 
}
\end{remark}

\begin{corollary} \label{Corollary_OptimalityCond}
Let $f$ and $x^*$ be as in the proposition above. Then the optimality condition \eqref{CodiffOptCond} holds true iff
$0 \in \underline{d} f(x^*) + (0, w)$ for all $(0, w) \in \extr \overline{d} f(x^*)$, where ``$\extr$'' stands
for the set of extreme points of a convex set.
\end{corollary}

\begin{proof}
The validity of ``only if'' part of the statement is obvious; therefore, let us prove the ``if'' part. It is easy to
verify that $(0, w) \in \extr \overline{d} f(x^*)$ iff $w \in \extr \overline{\partial} f(x^*)$. Hence and from the
Krein-Milman theorem (recall that the set $\overline{\partial} f(x^*)$ is nonempty, weakly compact and convex) it
follows that there exists at least one point $(0, w) \in \extr \overline{d} f(x^*)$.

Arguing by reductio ad absurdum, suppose that \eqref{CodiffOptCond} does not hold true. Then there exists 
$w \in \overline{\partial} f(x^*)$ such that $0 \notin \underline{\partial} f(x^*) + w$. By the separation theorem
there exist $h \in \mathcal{H}$ and $c > 0$ such that $\langle v + w, h \rangle \ge c$ for all
$v \in \underline{\partial} f(x^*)$. Applying the Krein-Milman theorem one obtains that there exists 
$w_0 \in \co \extr \overline{\partial} f(x^*)$ such that $|\langle w - w_0, h \rangle| < c / 2$ (recall that 
$\overline{\partial} f(x^*) = \cl \co \extr \overline{\partial} f(x^*)$, where the closure is taken in the weak
topology due to the fact that the space $\mathcal{H}$ is infinite dimensional in the general case). Therefore
$\langle v + w_0, h \rangle \ge c / 2$ for all $v \in \underline{\partial} f(x^*)$, which implies that
$0 \notin \underline{\partial} f(x^*) + w_0$. On the other hand, from the fact that 
$w_0 \in \co \extr \overline{\partial} f(x^*)$ it obviously follows that 
$0 \in \underline{\partial} f(x^*) + w_0$, which is impossible. Thus, the proof is complete.	 
\end{proof}

With the use of the optimality condition \eqref{CodiffOptCond} we can prove the mean value theorem for codifferentiable
functions. The proof of this result almost literally repeats the proof of the classical mean value theorem for
differentiable functions. To the best of author's knowledge, this proof has never been published before.

\begin{proposition}
Let $U \subset \mathcal{H}$ be an open set, and let $f \colon U \to \mathbb{R}$ be codifferentiable on a set 
$C \subseteq U$. Then for any points $x_1, x_2 \in C$ with $\co\{ x_1, x_2 \} \subseteq C$ there exist
$y \in \co\{ x_1, x_2 \}$, $(0, v) \in \underline{d} f(y)$, and $(0, w) \in \overline{d} f(y)$ such that
$f(x_2) - f(x_1) = \langle v + w, x_2 - x_1 \rangle$.
\end{proposition}

\begin{proof}
Let a function $g \colon (- \varepsilon, 1 + \varepsilon) \to \mathbb{R}$, where $\varepsilon > 0$, be codifferentiable
on $[0, 1]$. Applying the definition of codifferentiable function one can easily verify that $g$ is continuous on 
$[0, 1]$.

Suppose, at first, that $g(0) = g(1) = 0$. Then $g$ attains either a local minimum or a local maximum at a point 
$\theta \in (0, 1)$. If $g$ attains a local minimum at $\theta$, then applying Proposition~\ref{Prp_NessOptCond} one
obtains that for all $(0, w) \in \overline{d} g(\theta)$ there exists $(0, v) \in \underline{d} g(\theta)$ such that 
$v + w = 0$. If $g$ attains a local maximum at $\theta$, then applying Proposition~\ref{Prp_NessOptCond} to 
the function $-g$ one gets that for all $(0, v) \in \underline{d} g(\theta)$ there exists 
$(0, w) \in \overline{d} g(\theta)$ such that $v + w = 0$ (note that the function $\eta(t) = - g(t)$ is obviously
codifferentiable, and $D \eta(\theta) = [- \overline{d} g(\theta), - \underline{d} g(\theta)]$). In either case there
exist $(0, v) \in \underline{d} g(\theta)$ and $(0, w) \in \overline{d} g(\theta)$ such that $v + w = 0$.

Suppose, now, that $g(0)$ and $g(1)$ are arbitrary. For any $t \in (- \varepsilon, 1 + \varepsilon)$ define 
$r(t) = g(t) - g(0) - t (g(1) - g(0))$. Then $r(0) = r(1) = 0$ and, as it is easily seen, the function $r$ is
codifferentiable on $[0, 1]$, and one can define 
$D r(t) = [\underline{d} g(t) + (0, g(0) - g(1)), \overline{d} g(t)]$. Hence by the first part of the proof there exist
$\theta \in (0, 1)$, $(0, v) \in \underline{d} g(\theta)$, and $(0, w) \in \overline{d} g(\theta)$ such that
$g(1) - g(0) = v + w$.

Let us, now, return to the function $f$. Let points $x_1, x_2 \in C$ be such that $\co\{ x_1, x_2 \} \subset C$. Since
$U$ is open and $C \subseteq U$, there exists $\varepsilon > 0$ such that for any 
$t \in (- \varepsilon, 1 + \varepsilon)$ one has $x(t) = x_1 + t (x_2 - x_1) \in U$. Define $g(t) = f(x(t))$ for 
all $t \in (- \varepsilon, 1 + \varepsilon)$. One can easily verify that the function $g$ is codifferentiable 
on $[0, 1]$ (see, e.g., \cite[Thm.~4.5]{Dolgopolik_AbstractConv}), and one can define
\begin{gather*}
  \underline{d} g(t) = \big\{ (a, \langle v, x_2 - x_1 \rangle) \bigm| (a, v) \in \underline{d} f(x(t)) \big\}, \\
  \overline{d} g(t) = \big\{ (b, \langle w, x_2 - x_1 \rangle) \bigm| (b, w) \in \overline{d} f(x(t)) \big\}.
\end{gather*}
Hence by the second part of the proof there exist $\theta \in (0, 1)$, $(0, v) \in \underline{d} f(x(\theta))$, and
$(0, w) \in \overline{d} f(x(\theta))$ such that $\langle v + w, x_2 - x_1 \rangle = g(1) - g(0) = f(x_2) - f(x_1)$,
which completes the proofs.	 
\end{proof}

\begin{remark}
Strictly speaking, in the statement of the proposition above one must indicate that this proposition
holds true for any codifferential mapping $D f$ on the set $C$. However, in order not to overcomplicate the statement of
the proposition, as well as the statements of all results below, we implicitly mean that each result below holds
true for all codifferential mappings satisfying the assumptions of this result.
\end{remark}

\begin{corollary} \label{Corollary_LipschitzContinuity}
Let $U \subset \mathcal{H}$ be an open set, and let a function $f \colon U \to \mathbb{R}$ be codifferentiable on a
convex set $C \subseteq U$. Suppose also that
$$
  R = \sup_{x \in C} \sup\big\{ \| v \| \bigm| (0, v) \in \underline{d} f(x) + \overline{d} f(x) \big\} < + \infty.
$$
Then $f$ is Lipschitz continuous on $C$ with a Lipschitz constant $L = R$. In particular, if $f$ is continuously
codifferentiable on $U$, then $f$ is locally Lipschitz continuous on this set.
\end{corollary}

\begin{remark}
{(i)~Let us note that a different proof of the fact that any continuously codifferentiable function is locally Lipschitz
continuous was given in \cite{Kuntz} in the finite dimensional case.
}

\noindent{(ii)~Since a continuously codifferentiable function $f$ is locally Lipschitz continuous, one can consider its
Clarke subdifferential. Let us note that the Clarke subdifferential of the function $f$ is connected with a
quasidifferential of this function via the Demyanov difference (see~\cite[Chapter~III.4]{DemRub_book}). Furthermore, it
is easy to see that if $\mathcal{H}$ is finite dimensional, and $f$ is hypodifferentiable and regular (in the sense of
Clarke; see~\cite{Clarke}) at $x$, then its Clarke subdifferential at $x$ has 
the form $\partial_{Cl} f(x) = \{ v \in \mathcal{H} \mid (0, v) \in \underline{d} f(x) \}$.
}
\end{remark}

Let a function $f$ be defined in a neighbourhood $U$ of a point $x \in \mathcal{H}$ and codifferentiable at this point.
For any $r > 0$ such that $B(x, r) \subset U$, and for all $\Delta x \in B(0, r)$ and $\alpha \in [0, 1]$ define
\begin{equation} \label{ApproximMeasure}
\begin{split}
  \varepsilon_f(\alpha, \Delta x, x, r) = \frac{1}{\alpha}
  \Big( f(x + \alpha \Delta x) - f(x) 
  &- \max_{(a, v) \in \underline{d} f(x)} \big( a + \alpha \langle v, \Delta x \rangle \big) \\
  &- \min_{(b, w) \in \overline{d} f(x)} \big( b + \alpha \langle w, \Delta x \rangle \big) \Big)
\end{split}
\end{equation}
Obviously, the function $\varepsilon_f$ depends on the choice of a codifferential $D f(x)$. However, in order not to
overcomplicate the notation, we do not indicate this dependence explicitly. Note that from the definition of
codifferentiable function it follows that $\varepsilon_f(\alpha, \Delta x, x, r) \to 0$ as $\alpha \to +0$. In
essence, the function $\varepsilon_f$ measures how well the codifferential $D f(x)$ approximates the function $f$ at the
point $x$ in a direction $\Delta x \in B(0, r)$.

Let us derive main calculus rules for continuously codifferentiable functions along with some simple estimates of the
function $\varepsilon_f$. These results will be important for understanding the theorems on convergence of the method of
codifferential descent presented in the following sections. Let us note that although the codifferential calculus is
well-developed (see, e.g., \cite{DemRub_book,Dolgopolik,Dolgopolik_AbstractConv}), the estimates of the function
$\varepsilon_f$ obtained below have never been published before.

For any pairs $[A, B]$ and $[C, D]$ of convex subsets $A$, $B$, $C$ and $D$ of a linear space $X$ define
$[A, B] + [C, D] = [A + C, B + D]$, and for all $\lambda \in \mathbb{R}$ define
$$
  \lambda [A, B] = \begin{cases}
    [\lambda A, \lambda B], & \text{if } \lambda \ge 0, \\
    [\lambda B, \lambda A], & \text{if } \lambda < 0.
  \end{cases}
$$
Note that these rules of addition and multiplication by scalar are the same as in the
Minkowski-R\aa{}dst\"om-H\"ormander space (see, e.g., \cite{PallaschkeUrbnski}). Let $U \subset \mathcal{H}$ be an open
set.

\begin{theorem} \label{Theorem_Codiff_Max_Min}
Let functions $f_i \colon U \to \mathbb{R}$, $i \in I = \{ 1, \ldots, m \}$ be continuously codifferentiable on $U$.
Then the functions $f = \max_{i \in I} f_i$ and $g = \min_{i \in I} f_i$ are continuously codifferentiable on $U$ as
well, and for any $x \in U$ one can define
\begin{gather} \label{MaxFuncCodiff}
  D f(x) = \Big[
  \co\Big\{ \{ (f_i(x) - f(x), 0) \} + \underline{d} f_i(x) - \sum_{j \ne i} \overline{d} f_j(x) \Bigm| i \in I \Big\},
  \sum_{i \in I} \overline{d} f_i(x) \Big], \\
  D g(x) = \Big[ \sum_{i \in I} \underline{d} f_i(x),
  \co\Big\{ \{ (f_i(x) - g(x), 0) \} + \overline{d} f_i(x) - \sum_{j \ne i} \underline{d} f_j(x) \Bigm| i \in I \Big\}
  \Big]. \notag
\end{gather}
Furthermore, one has 
\begin{equation} \label{MaxFunc_ApproxMeasure}
  \max\big\{ \big| \varepsilon_f(\alpha, \Delta x, x, r) \big|, \big| \varepsilon_g(\alpha, \Delta x, x, r) \big| \big\}
  \le \max_{i \in I} \big| \varepsilon_{f_i}(\alpha, \Delta x, x, r) \big|.
\end{equation}
\end{theorem}

\begin{proof}
We only consider the function $f$, since the proof for the function $g$ is almost the same. Fix an arbitrary $x \in U$,
and let $r > 0$ be such that $B(x, r) \subset U$. Denote
$$
  \Phi_i(y) = \max_{(a, v) \in \underline{d} f_i(x)} \big( a + \langle v, y \rangle \big), \quad
  \Psi_i(y) = \min_{(b, w) \in \overline{d} f_i(x)} \big( b + \langle w, y \rangle \big),
$$
and define
$$
  \Phi(y) = \max_{(a, v) \in \underline{d} f(x)} \big( a + \langle v, y \rangle \big), \quad
  \Psi(y) = \min_{(b, w) \in \overline{d} f(x)} \big( b + \langle w, y \rangle \big),
$$
where the pair $D f(x) = [\underline{d} f(x), \overline{d} f(x)]$ is defined in \eqref{MaxFuncCodiff}. Taking
into account the fact that the functions $f_i$ are locally Lipschitz continuous on $U$
by Corollary~\ref{Corollary_LipschitzContinuity} one can easily verify that the mappings $\underline{d} f(\cdot)$ and
$\overline{d} f(\cdot)$ are Hausdorff continuous on $U$. Moreover, the sets $\underline{d} f(x)$ and $\overline{d} f(x)$
are obviously weakly compact and convex.

Fix arbitrary $\Delta x \in B(0, r)$ and $\alpha \in [0, 1]$. By definition one has
\begin{multline} \label{MaxFunc_CodiffDecomposition}
  f(x + \alpha \Delta x) - f(x) = \max_{i \in I} \big( f_i(x + \alpha \Delta x) - f(x) \big)
  = \max_{i \in I} \big( f_i(x) - f(x) \\
  + \Phi_i(\alpha \Delta x) + \Psi_i(\alpha \Delta x) + \alpha \varepsilon_{f_i}(\alpha, \Delta x, x, r) \big)
  = \max_{i \in I} \big( f_i(x) - f(x) \\
  + \Phi_i(\alpha \Delta x) - \sum_{j \ne i} \Psi_i(\alpha \Delta x)
  + \alpha \varepsilon_{f_i}(\alpha, \Delta x, x, r) \big) + \sum_{i = 1}^m \Psi_i(\alpha \Delta x).
\end{multline}
From the obvious equality 
$$
  \Phi(y) = \max_{i \in I} 
  \big( f_i(x) - f(x) + \Phi_i(y) - \sum_{j \ne i} \Psi_i(y) \big)
$$
it follows that
\begin{multline*}
  \Big| \max_{i \in I} \Big( f_i(x) - f(x) + \Phi_i(\alpha \Delta x) - \sum_{j \ne i} \Psi_i(\alpha \Delta x)
  + \alpha \varepsilon_{f_i}(\alpha, \Delta x, x, r) \Big) \\
  - \Phi(\alpha \Delta x) \Big| \le
  \alpha \max_{i \in I} \big| \varepsilon_{f_i}(\alpha, \Delta x, x, r) \big|.
\end{multline*}
Hence applying \eqref{MaxFunc_CodiffDecomposition} and the fact that $\Psi(y) = \sum_{i \in I} \Psi_i(y)$ one obtains
that $| \varepsilon_f(\alpha, \Delta x, x, r) | \le \max_{i \in I} | \varepsilon_{f_i}(\alpha, \Delta x, x, r) |$,
where $\varepsilon_f(\alpha, \Delta x, x, r)$ is defined as in \eqref{ApproximMeasure}. Therefore the function $f$ is
codifferentiable at the point $x$, $D f(x)$ is a codifferential of $f$ at this point, and inequality
\eqref{MaxFunc_ApproxMeasure} holds true.	 
\end{proof}

\begin{theorem} \label{Theorem_Codiff_Composition}
Let functions $f_i \colon U \to \mathbb{R}$, $i \in I = \{ 1, \ldots, m \}$ be continuously codifferentiable on $U$, and
a real-valued function $G$ be defined and continuously differentiable on an open set $V$ containing the set
$\{ f(x) \in \mathbb{R}^m \mid x \in U \}$, where $f(x) = (f_1(x), \ldots, f_m(x))$. Then the function  
$g(\cdot) = G(f_1(\cdot), \ldots, f_m(\cdot))$ is continuously codifferentiable on $U$, and for any $x \in U$ one can
define
\begin{equation} \label{CompositionCodiff}
  D g(x) = \sum_{i = 1}^m \frac{\partial G}{\partial y_i}(f(x)) D f_i(x).
\end{equation}
Moreover, for any $r > 0$ such that $B(x, r) \subset U$ and $\co f(B(x, r)) \subset V$ one has
\begin{multline} \label{Composition_ApproxMeasure}
  \big| \varepsilon_g (\alpha, \Delta x, x, r) \big| \le
  \sum_{i = 1}^m \left| \frac{\partial G}{\partial y_i}(f(x)) \right| 
  \big| \varepsilon_{f_i} (\alpha, \Delta x, x, r) \big| \\
  + \frac{1}{\alpha} 
  \sup_{t \in [0, 1]} \Big| \big\langle \nabla G\big( y(t) \big) - \nabla G(f(x)),
   f(x + \alpha \Delta x) - f(x) \big\rangle \Big|,
\end{multline}
where $\langle \cdot, \cdot \rangle$ is the inner product in $\mathbb{R}^m$, 
and $y(t) = t f(x) + (1 - t) f(x + \alpha \Delta x)$.
\end{theorem}

\begin{proof}
Fix arbitrary $x \in U$ and $r > 0$ such that $B(x, r) \subset U$ and $\co f(B(x, r)) \subset V$. Note that such $r > 0$
exists due to the fact that $f_i$ are locally Lipschitz continuous on $U$ by
Corollary~\ref{Corollary_LipschitzContinuity}. Choose $\Delta x \in B(0, r)$ and $\alpha \in [0, 1]$, and denote 
$y(t) = t f(x) + (1 - t) f(x + \alpha \Delta x)$. Applying the mean value theorem one obtains that 
there exists $t \in (0, 1)$ such that
\begin{multline*}
  g(x + \alpha \Delta x) - g(x) = \langle \nabla G(f(x)), f(x + \alpha \Delta x) - f(x) \rangle \\
  + \big\langle \nabla G\big( y(t) \big) - \nabla G(f(x)), f(x + \alpha \Delta x) - f(x) \big\rangle \\
  = \sum_{i = 1}^m \frac{\partial G}{\partial y_i}(f(x)) 
  \Big( \Phi_i(\alpha \Delta x) + \Psi_i(\alpha \Delta x) + \alpha \varepsilon_{f_i}(\alpha, \Delta x, x, r) \Big) \\
  + \big\langle \nabla G\big( y(t) \big) - \nabla G(f(x)), f(x + \alpha \Delta x) - f(x) \big\rangle,
\end{multline*}
where $\Phi_i$ and $\Psi_i$ are defined as in the proof of Theorem~\ref{Theorem_Codiff_Max_Min}. Hence taking into
account the fact that the functions $f_i$ are locally Lipschitz continuous one can easily verify that the function $g$
is codifferentiable at $x$, the pair $D g(x)$ defined in \eqref{CompositionCodiff} is a codifferential of $g$ at $x$,
and \eqref{Composition_ApproxMeasure} holds true. It remains to note that codifferential mapping
\eqref{CompositionCodiff} is obviously continuous.	 
\end{proof}

\begin{corollary} \label{Corollary_Codiff_Lin_Combin}
Let functions $f_i \colon U \to \mathbb{R}$, $i \in I = \{ 1, \ldots, m \}$ be continuously codifferentiable on $U$.
Then for any $\lambda_i \in \mathbb{R}$, $i \in I$ the function $f = \sum_{i \in I} \lambda_i f_i$ is continuously
codifferentiable on $U$, and
$$
  D f(x) = \sum_{i \in I} \lambda_i D f_i(x), \quad
  \big| \varepsilon_f(\alpha, \Delta x, x, r) \big| \le 
  \sum_{i \in I} |\lambda_i| \cdot \big| \varepsilon_{f_i}(\alpha, \Delta x, x, r) \big|.
$$
\end{corollary}

\begin{corollary} \label{Corollary_Codiff_Multiplication}
Let functions $f_1, f_2 \colon U \to \mathbb{R}$ be continuously codifferentiable on $U$. Then the function 
$f = f_1 \cdot f_2$ is continuously codifferentiable on $U$, for any $x \in U$ one can define
$D f(x) = f_2(x) D f_1(x) + f_1(x) D f_2(x)$, and for any sufficiently small $r > 0$ one has
\begin{multline*}
  \big| \varepsilon_f(\alpha, \Delta x, x, r) \big| \le 
  |f_2(x)| \cdot \big| \varepsilon_{f_1} (\alpha, \Delta x, x, r) \big|
  + |f_1(x)| \cdot \big| \varepsilon_{f_2}(\alpha, \Delta x, x, r) \big| \\
  + 2 \alpha L_1 L_2 \| \Delta x \|^2,
\end{multline*}
where $L_i > 0$ is a Lipschitz constant of $f_i$ on $B(x, r)$.
\end{corollary}

\begin{corollary} \label{Corollary_Codiff_Division}
Let a function $f \colon U \to \mathbb{R}$ be continuously codifferentiable on $U$, and $f(x) \ne 0$ for all $x \in U$.
Then the function $g = 1 / f$ is continuously codifferentiable on $U$, for any $x \in U$ one can define
$D g(x) = - f(x)^{-2} D f(x)$, and for any sufficiently small $r > 0$ one has
$$ 
  \big| \varepsilon_g(\alpha, \Delta x, x, r) \big| \le 
  \frac{1}{f(x)^2} \big| \varepsilon_f(\alpha, \Delta x, x, r) \big|
  + \frac{2 \alpha L^2 \| \Delta x \|^2}{\min\{ |f(x) - r L|^3, |f(x) + r L|^3 \}},
$$
where $L > 0$ is a Lipschitz constant of $f$ on $B(x, r)$.
\end{corollary}

\begin{remark}
Note that equalities \eqref{BothApproxZeroAtZero} hold true automatically, if one computes a codifferential of a
function $f$ with the use of the calculus rules presented above and
Examples~\ref{Example_Differentiable}--\ref{Example_ConvexFunction}.
\end{remark}

In the end of this section, let us introduce a class of \textit{uniformly codifferentiable} function that will
play a very important role in the convergence analysis of the method of codifferential descent. 

\begin{definition}
Let a function $f \colon U \to \mathbb{R}$ be (continuously) codifferentiable on $U$, and let a set $C \subset U$ be
such that there exists $r > 0$ for which $B(x, r) \subseteq U$ for all $x \in C$. One says that $f$ is
\textit{(continuously) uniformly codifferentiable} on a set $C \subseteq U$ if there exists a (continuous)
codifferential mapping $D f(\cdot) = [\underline{d} f(\cdot), \overline{d} f(\cdot)]$ with
$\underline{d} f(\cdot), \overline{d} f(\cdot) \colon U \rightrightarrows \mathbb{R} \times \mathcal{H}$ such that 
$\varepsilon_f(\alpha, \Delta x, x, r) \to 0$ as $\alpha \to + 0$ uniformly for all $\Delta x \in B(0, r)$ and 
$x \in C$. In this case one says that the (continuous) codifferential mapping $D f$ uniformly approximates the function
$f$ on the set $C$.
\end{definition}

One can easily check that if a function $f \colon U \to \mathbb{R}$ is G\^{a}teaux differentiable on $U$, and its 
G\^{a}teaux derivative is uniformly continuous on the set $\cup_{x \in C} B(0, \tau)$ for some $\tau > 0$, then $f$ is
continuously uniformly codifferentiable on the set $C$. In particular, if $\mathcal{H}$ is finite dimensional, then a
continuously differentiable function $f$ is continuously uniformly codifferentiable on any bounded set.

Observe that Theorems~\ref{Theorem_Codiff_Max_Min} and \ref{Theorem_Codiff_Composition} are very useful for
verifying whether a given codifferentiable function is uniformly codifferentiable on a given set. In particular,
Theorem~\ref{Theorem_Codiff_Max_Min} and Corollary~\ref{Corollary_Codiff_Lin_Combin} imply that the set of all
functions that are continuously uniformly codifferentiable on a given set $C$ is closed under addition, multiplication
by scalar, as well as the pointwise maximum and minimum of finite families of functions (i.e. this set is a vector
lattice). In turn, Corollaries~\ref{Corollary_Codiff_Multiplication} and \ref{Corollary_Codiff_Division} imply that 
the set of all locally (i.e. in a neighbourhood of every point) continuously uniformly codifferentiable functions is
closed under all standard algebraic operations. Furthermore, from Theorem~\ref{Theorem_Codiff_Composition} it follows
that the function $g(x) = G(f_1(x), \ldots, f_m(x))$ is continuously uniformly codifferentiable on a set $C \subset U$,
provided the functions $f_i$ are continuously uniformly codifferentiable and Lipschitz continuous on $U$, the function
$G$ is continuously differentiable on an open set $V$ containing the set $\cup_{x \in C} \co f(B(x, r))$ for some 
$r > 0$, and the gradient $\nabla G(y)$ is uniformly continuous on $V$. For example, the function
$$
  f(x_1, x_2) = \sin\big( \min\{ x_1, x_2 \} \big) \cdot e^{\max\{ x_1, x_2 \}} + \big| \sinh(x_1 + |x_2|) \big|
$$
is continuously uniformly codifferentiable on any bounded subset of $\mathbb{R}^2$. Moreover, with the use of 
Theorems~\ref{Theorem_Codiff_Max_Min} and \ref{Theorem_Codiff_Composition} one can easily compute a continuous
codifferential mapping $D f(\cdot)$ of the function $f$ on $\mathbb{R}^2$ uniformly approximating this function on any
bounded set.

\section{The method of codifferential descent}
\label{Section_MCD}

In this section we present a general scheme of the method of codifferential descent, and prove the global
convergence of this method.

\subsection{A description of the method}

Let $U \subseteq \mathcal{H}$ be an open set, and let a function $f \colon U \to \mathbb{R}$ be codifferentiable on
$U$. Hereinafter, we suppose that a codifferential mapping $D f(\cdot)$ of the function $f$ on the set $U$ is fixed.
Recall that if $x^* \in U$ is a point of local minimum of the function $f$, then
\begin{equation} \label{InfStatPoint_Def}
  0 \in \underline{d} f(x^*) + (0, w) \quad \forall (0, w) \in \extr \overline{d} f(x^*)
\end{equation}
by Corollary~\ref{Corollary_OptimalityCond}. Any point $x^* \in U$ satisfying \eqref{InfStatPoint_Def} is called
\textit{an inf-stationary} point of the function $f$. Note that by Proposition~\ref{Prp_NessOptCond} the set of
inf-stationary points of $f$ is independent of the choice of a codifferential. Let us describe a method for finding
inf-stationary points of the function $f$ on the set $U$ called \textit{the method of codifferential descent} (MCD).

\begin{remark}
Note that our aim is to minimize the function $f$ on an open set $U$, but not necessarily on the entire space
$\mathcal{H}$. To this end, below we suppose that a sequence generated by the MCD does not leave the set $U$. This
assumption might seem unnatural at first glance; however, it allows one to apply the results on the convergence of 
the MCD obtained in the article to ``barrier-like'' functions $f$, i.e. to those nonsmooth functions $f$ which are equal
to $+ \infty$ outset an open set $U$. In particular, the convergence results below can be applied to exact barrier
functions for nonsmooth optimization problems (see~\cite{ExactBarrierFunc}).
\end{remark}

For any $\nu, \mu \in [0, + \infty]$ let set-valued mappings
$\underline{d}_{\nu} f(\cdot), \overline{d}_{\mu} f(\cdot) \colon U \rightrightarrows \mathbb{R} \times \mathcal{H}$ be
such that for any $x \in U$ one has
\begin{equation} \label{ReducedCodifferentialDef}
\begin{split}
  \big\{ (a, v) \in \extr \underline{d}f(x) \bigm| a \ge - \nu \big\} &\subseteq \underline{d}_{\nu} f(x)
  \subseteq \underline{d} f(x), \\
  \big\{ (b, w) \in \extr \overline{d}f(x) \bigm| b \le \mu \big\} &\subseteq \overline{d}_{\mu} f(x)
  \subseteq \overline{d} f(x). 
\end{split}
\end{equation}
The pair $[\underline{d}_{\nu} f(x), \overline{d}_{\mu} f(x)]$ is called \textit{a truncated codifferential} of $f$ at
$x$, and it can be viewed as a kind of approximation of both quasidifferential $\mathscr{D} f(x)$ and codifferential $D
f(x)$ of $f$ at $x$, depending on the values of the parameters $\nu$ and $\mu$. Namely, one has
$$
  \{ 0 \} \times \extr \underline{\partial} f(x) \subseteq \underline{d}_{0} f(x), \quad
  \extr \underline{d} f(x) \subseteq \underline{d}_{\infty} f(x) \subseteq \underline{d} f(x),
$$
and similar relations hold true for $\overline{d}_{\mu} f(x)$. Let us note that in applications 
the set $\underline{d} f(x)$ is typically a convex hull of a finite set of points $(a_i, v_i)$. In this case one usually
defines $\underline{d}_{\nu} f(x)$ as the set of all those $(a_i, v_i)$ for which $a_i \ge - \nu$. The set
$\overline{d}_{\mu} f(x)$ is defined in a similar way.

\begin{remark}
In theory, the best possible choice of the sets $\underline{d}_{\nu} f(x)$ and $\overline{d}_{\mu} f(x)$ is 
\begin{equation*}
\begin{split}
  \underline{d}_{\nu} f(x) &= \big\{ (a, v) \in \extr \underline{d}f(x) \bigm| a \ge - \nu \big\}, \quad \\
  \overline{d}_{\mu} f(x) &= \big\{ (b, w) \in \extr \overline{d}f(x) \bigm| b \le \mu \big\}.
\end{split}
\end{equation*}
However, in order to define the truncated codifferential this way in practice, one has to find all extreme points of the
sets $\underline{d} f(x)$ and $\overline{d} f(x)$, which is a very computationally expensive procedure. That is why in
some application it might be more efficient to use larger sets $\underline{d}_{\nu} f(x)$ and $\overline{d}_{\mu} f(x)$,
but avoid the search of extreme points. The main goal of this article is to analyse the convergence of the MCD in the
general case. That is why we do not specify the way the truncated codifferential is defined, and only impose assumptions
\eqref{ReducedCodifferentialDef} that are somewhat necessary to ensure convergence.
\end{remark}

Let the space $\mathbb{R} \times \mathcal{H}$ be equipped with the inner product 
$\langle (a, v), (b, w) \rangle = a b + \langle v, w \rangle$ and the corresponding norm. The scheme of the method of
codifferential descent is as follows.

\begin{enumerate}
\item{Choose sequences $\{ \nu_n \}, \{ \mu_n \} \subset [0, + \infty]$, the upper bound $\alpha_* \in (0, + \infty]$
on the step size, and an initial point $x_0 \in U$.
}
\item{$n$th iteration ($n \ge 0$).
  \begin{enumerate}
    \item{Compute $\underline{d}_{\nu_n} f(x_n)$ and $\overline{d}_{\mu_n} f(x_n)$.
    }
    \item{For any $z = (b, w) \in \overline{d}_{\mu_n} f(x_n)$ compute
    $$
      \{ (a_n(z), v_n(z)) \} = \argmin\Big\{ \| (a, v) \| \Bigm| 
	 (a, v) \in \cl \co \underline{d}_{\nu_n} f(x_n) + z \Big\}.
    $$
    }
    \item{For any $z \in \overline{d}_{\mu_n} f(x_n)$ compute
    $$
      \alpha_n(z) \in \argmin\Big\{ f(x_n - \alpha v_n(z)) \Bigm| 
      \alpha \in [0, \alpha_*] \colon x_n - \alpha v_n(z) \in U \Big\}.
    $$
    }
    \item{Compute
    $$
      z_n \in \argmin\Big\{ f(x_n - \alpha_n(z) v_n(z)) \Bigm| z \in \overline{d}_{\mu_n} f(x_n) \Big\},
    $$
    and define $x_{n + 1} = x_n - \alpha_n(z_n) v_n(z_n)$.
    }
  \end{enumerate}
}
\end{enumerate}

\begin{remark}
Here the closure is taken in the norm (or, equivalently, weak) topology, and in the case $\alpha_* = + \infty$ we define
$[0, \alpha_*] = [0, + \infty)$. Note also that the set $\cl \co \underline{d}_{\nu_n} f(x_n) \subseteq \underline{d}
f(x_n)$ is weakly compact due to the weak compactness of the hypodifferential $\underline{d} f(x_n)$, which implies that
the pairs $(a_n(z), v_n(z))$ are correctly defined.
\end{remark}

Observe that at each iteration of the MCD one must perform line search in \textit{several} directions. As we will show
below, at least one of those directions is a descent direction (i.e. $f'(x_n, -v_n(z)) < 0$), provided $x_n$ is not an
inf-stationary point of the function $f$. Therefore, for any $n \in \mathbb{N}$ either $x_n$ is an inf-stationary point
of $f$ or $f(x_{n + 1}) < f(x_n)$. On the other hand, some of the search directions $- v_n(z)$ might not be descent
directions of the function $f$, i.e. $f$ may first increase and then decrease in these directions. This interesting
feature of the MCD allows it to ``jump over'' some points of local minimum, provided the parameters $\nu_n$ and $\mu_n$
are sufficiently large (for a particular example of this phenomenon see~\cite{DemBagRub}, Sect.~4). Note that the fact
that some of the search directions $v_n(z)$ are not descent direction forces one to define step sizes via 
the minimization of the function $f$ along the directions $v_n(z)$ instead of utilizing some more widespread step size
rules, such as the Armijo and the Goldstein rules.

The parameters $\nu_n, \mu_n \ge 0$  are introduced into the MCD in order to ensure convergence. If one looks at the
form of the necessary optimality condition \eqref{InfStatPoint_Def}, then it might seem natural to utilize the MCD with
$\nu_n \equiv \mu_n \equiv 0$. However, the MCD with $\nu_n \equiv \mu_n \equiv 0$ might converge to a non-stationary
point of the function $f$ (cf.~\cite{DemyanovMalozemov}, Section~III.5, and \cite{HiriartUrrutyLemarechal},
Section~VIII.2.2). Let us also note that the choice of parameters $\nu_n$ and $\mu_n$ is a tradeoff between the
complexity of every iteration and the overall performance of the method. In many application, a decrease of the
parameters $\nu_n$ and $\mu_n$ reduces the cost of an iteration, while an increase of these parameters might allow
one to find a better local solution.

Finally, from this point onwards we suppose that the step sizes $\alpha_n(z)$ and the vector $z_n$ are well-defined in
every iteration of the MCD. The vector $z_n$ is well-defined, provided the sets $\overline{d}_{\mu_n} f(x_n)$ are
finite, which is the case in almost all applications. In turn, the step sizes $\alpha_n(z)$ are well-defined if $f$ is
l.s.c., $U = \mathcal{H}$ (or $f(x) = + \infty$ outside $U$) and $\alpha_* < + \infty$. In the case 
$\alpha_* = + \infty$, one must make an additional assumption on the function $f$ such as the boundedness of the set 
$\{ x \in U \mid f(x) < f(x_0) \}$.

\begin{remark} \label{Remark_MCD_Unrealisable}
Let us note that the original version of the MCD \cite{DemRub_book} corresponds to the case when $\alpha_* = + \infty$,
$\mu_n \equiv \mu > 0$, $\underline{d}_{\nu_n} f(x_n) \equiv \underline{d} f(x_n)$ and
$\overline{d}_{\mu_n} f(x_n) \equiv \{ (b, w) \in \overline{d} f(x_n) \mid b \le \mu \}$ for some $\mu > 0$. However,
a direct practical implementation of the original method is impossible, since the set 
$\{ (b, w) \in \overline{d} f(x) \mid b \le \mu \}$ almost always has the cardinality of the continuum (unless $f$ is
hypodifferentiable). Note also that a direct implementation of the method of truncated codifferential \cite{DemBagRub}
is impossible for a similar reason. The version of the MCD proposed in this paper is, in fact, a mathematical
description of the way the original method and the method of truncated codifferential are
implemented in practice. 
\end{remark}

\subsection{Auxiliary Results}

Before we proceed to the convergence analysis of the MCD, let us first prove a simple result about descent directions in
the MCD, and two useful auxiliary lemmas that will be utilized throughout the rest of the article.

\begin{lemma} \label{Lemma_DescentDirections}
Let a sequence $\{ x_n \} \subset U$ be generated by the MCD. Then for any $n \in \mathbb{N}$ and 
$z = (0, w) \in \overline{d}_{\mu_n} f(x_n)$ such that $0 \notin \cl\co\underline{d}_{\nu_n} f(x_n) + z$ one has
$f'(x_n, -v_n(z)) \le - \| (a_n(z), v_n(z)) \|^2$, which implies that $v_n(z) \ne 0$, provided $a_n(z) \ne 0$. Moreover,
if $x_n$ is not an inf-stationary point of the function $f$, then $f(x_{n + 1}) < f(x_n)$.
\end{lemma}

\begin{proof}
Fix arbitrary $n \in \mathbb{N}$ and $z = (0, w_0) \in \overline{d}_{\mu_n} f(x_n)$ such that 
$0 \notin \cl \co \underline{d}_{\nu_n} f(x_n) + z$. Applying the necessary condition for a minimum of a convex
function on a convex set one obtains that
\begin{equation} \label{SeparationTheorem}
  a a_n(z) + \langle v, v_n(z) \rangle \ge \big\| (a_n(z), v_n(z)) \big\|^2 
  \quad \forall (a, v) \in \cl \co \underline{d}_{\nu_n} f(x_n) + z.
\end{equation}
Note that by definition one has 
$\{ 0 \} \times \underline{\partial} f(x_n) \subseteq \cl\co \underline{d}_{\nu_n} f(x_n)$
and $w_0 \in \overline{\partial} f(x_n)$. Therefore
\begin{multline*}
  f'(x_n, -v_n(z)) = \max_{v \in \underline{\partial} f(x_n)} \langle v, - v_n(z) \rangle 
  + \min_{w \in \overline{\partial} f(x_n)} \langle w, - v_n(z) \rangle \\
  \le \max_{v \in \underline{\partial} f(x_n) + w_0} \langle v, - v_n(z) \rangle \le - \big\| (a_n(z), v_n(z)) \big\|^2.
\end{multline*}
Finally, note that if $x_n$ is not an inf-stationary point of $f$, then by definition there 
exists $z = (0, w) \in \overline{d}_{\mu_n} f(x_n)$ such that 
$0 \notin \cl \co \underline{d}_{\nu_n} f(x_n) + z$, which implies that $f(x_n - \alpha_n(z) v_n(z)) < f(x_n)$,
and hence $f(x_{n + 1}) < f(x_n)$.	 
\end{proof}

\begin{lemma}\label{Lemma_LimitOfExtremePoints}
Let $X$ be a finite dimensional normed space, and let a sequence $\{ A_n \}$ of convex compact subsets of $X$ converge
to a convex compact set $A \subset X$ in the Hausdorff metric. Then for any subsequence $\{ A_{n_k} \}$ one has
$\extr A \subseteq \limsup_{k \to \infty} \extr A_{n_k}$, where $\limsup$ is the outer limit.
\end{lemma}

\begin{proof}
Let us verify that $A = \co ( \limsup_{k \to \infty} \extr A_{n_k} )$. Then applying a partial converse to the
Krein-Milman theorem (if $K$ is a compact convex set, and $K = \cl \co B$, then $\extr K \subseteq \cl B$; see, e.g.,
\cite[Proposition~10.1.3]{Edwards}), and the fact that the outer limit set is always closed we arrive at the required
result.

From the fact that $A_n \to A$ in the Hausdorff metric it obviously follows that 
$\co (\limsup_{k \to \infty} \extr A_{n_k}) \subseteq A$. Let us prove the converse inclusion. 
Fix an arbitrary $x \in A$. Clearly, for any $k \in \mathbb{N}$ there exists $x_k \in A_{n_k}$ such that $x_k \to x$ as
$k \to \infty$. Applying the Krein-Milman and the Carath\'{e}odory theorems one obtains that for any $k \in \mathbb{N}$
there exist $y^i_k \in \extr A_{n_k}$, and $\alpha^i_k \ge 0$, $1 \le i \le m + 1$ (here $m$ is the dimension of $X$)
such that
\begin{equation} \label{KreinMilman_plus_Caratheodory}
  x_k = \sum_{i = 1}^{m + 1} \alpha^i_k y^i_k, 
  \qquad \sum_{i = 1}^{m + 1} \alpha^i_k = 1.
\end{equation}
Note that the sequence $\{ A_n \}$ lies within a bounded set by virtue of the fact that $A$ is compact, and $A_n \to A$
in the Hausdorff metric. Hence the sequences $\{ y^i_k \}$, $i \in \{ 1, \ldots m + 1 \}$ are bounded. Replacing them as
well as the sequences $\{ \alpha_k^i \}$ with convergent subsequences, and passing to the limit in
\eqref{KreinMilman_plus_Caratheodory} one obtains that there exists $y^i \in \limsup_{k \to \infty} \extr A_{n_k}$ and
$\alpha^i \ge 0$ such that
$$
  x = \sum_{i = 1}^{m + 1} \alpha^i y^i, \qquad \sum_{i = 1}^{m + 1} \alpha^i = 1.
$$
Thus, $x \in \co (\limsup_{k \to \infty} \extr A_{n_k})$, and the proof is complete.	 
\end{proof}

\begin{lemma} \label{Lemma_GlobalDivergence}
Let sequences $\{ x_n \} \subset U$ and $\{ \nu_n \}, \{ \mu_n \} \subset [0, + \infty]$ be such that
\begin{enumerate}
\item{there exists $r > 0$ for which $B(x_n, r) \subset U$ for all $n \in \mathbb{N}$;}

\item{the codifferential mapping $D f$ uniformly approximates $f$ on the set $\{ x_n \}_{n \in \mathbb{N}}$;}

\item{the sequences $\{ \underline{d} f(x_n) \}$ and $\{ \overline{d} f(x_n) \}$ lie within a bounded set;}

\item{$\liminf_{n \to \infty} \nu_n = \nu^* > 0$.} 
\end{enumerate}
Suppose also that $\{ h_n \} \subset \mathcal{H}$ is a bounded sequence satisfying the inequalities
\begin{gather} \label{DescentCond1}
  \sup_{(a, v) \in \underline{d}_{\nu_n} f(x_n) + z_n} ( a + \langle v, h_n \rangle ) \le - \theta
  \quad \forall n \in \mathbb{N}, \\
  f(x_{n + 1}) \le \inf_{\alpha \in [0, \alpha_0]} f(x_n + \alpha h_n) + \varepsilon_n 
  \quad \forall n \in \mathbb{N}	\label{DescentCond2}
\end{gather}
for some $z_n = (b_n, w_n) \in \overline{d}_{\mu_n} f(x_n)$, $\theta > 0$, $\alpha_0 > 0$ and $\varepsilon_n$ such that
$\varepsilon_n \to 0$ and $b_n \to 0$ as $n \to \infty$. Then $f(x_n) \to - \infty$ as $n \to \infty$.
\end{lemma}

\begin{proof}
By definition for any $n \in \mathbb{N}$ one has
\begin{multline} \label{MCDTechResult_UCA}
  f(x_n + \alpha h_n) - f(x_n) = 
  \max_{(a, v) \in \underline{d} f(x_n)} \big( a + \alpha \langle v, h_n \rangle ) 
  + \min_{(b, w) \in \overline{d} f(x_n)} \big( b + \alpha \langle w, h_n \rangle ) \\
  + \alpha \varepsilon_n(\alpha)
  \le \max_{(a, v) \in \underline{d} f(x_n) + z_n} \big( a + \alpha \langle v, h_n \rangle )
  + \alpha \varepsilon_n(\alpha),
\end{multline}
where $\varepsilon_n(\alpha) = \varepsilon_f\left( \| h_n \| \alpha / r, \Delta x_n, x_n, r \right)$,
$\Delta x_n = r h_n / \| h_n \| \in B(x_n, r) \subset U$, and 
$0 \le \alpha \le \widehat{\alpha} := \min\{ 1, r / \sup_{n} \| h_n \| \}$ is arbitrary 
(note that $\widehat{\alpha} > 0$ due to the fact that the sequence $\{ h_n \}$ is bounded). Our aim is to prove that
there exist $n_1 \in \mathbb{N}$ and $\alpha_1 \in (0, \widehat{\alpha}]$ such that for all $n \ge n_1$ and 
$\alpha \in [0, \alpha_1]$ the set $\underline{d} f(x_n) + z_n$ in \eqref{MCDTechResult_UCA} can be replaced by the
smaller set $\underline{d}_{\nu_n} f(x_n) + z_n$, i.e.
\begin{equation} \label{MCDTechResult_WorseUCA}
  f(x_n + \alpha h_n) - f(x_n) \le 
  \sup_{(a, v) \in \underline{d}_{\nu_n} f(x_n) + z_n} \big( a + \alpha \langle v, h_n \rangle)
  + \alpha \varepsilon_n(\alpha)
\end{equation}
for all $n \ge n_1$ and $\alpha \in (0, \alpha_1]$. 

Before we turn to the proof of inequality \eqref{MCDTechResult_WorseUCA}, let us first demonstrate that the validity of
the lemma follows directly from this inequality. Indeed, denote
$$
  \eta_n(\alpha) = \sup_{(a, v) \in \underline{d}_{\nu_n} f(x_n) + z_n} \big( a + \alpha \langle v, h_n \rangle \big).
$$
From \eqref{MCDTechResult_WorseUCA} and the convexity of $\eta_n(\alpha)$ it follows that
\begin{multline*}
  f(x_n + \alpha h_n) - f(x_n) \le \alpha \eta_n(1) + (1 - \alpha) \eta_n(0) + \alpha \varepsilon_n(\alpha) \\
  = \alpha \sup_{(a, v) \in \underline{d}_{\nu_n} f(x_n) + z_n} \big( a + \langle v, h_n \rangle)
  + (1 - \alpha) \eta_n(0) + \alpha \varepsilon_n(\alpha)
\end{multline*}
for all $n \ge n_1$ and $\alpha \in [0, \alpha_1]$. Hence applying \eqref{DescentCond1} and taking into account the fact
that $\eta_n(0) = b_n$ due to \eqref{BothApproxZeroAtZero} one obtains that
$$
  f(x_n + \alpha h_n) - f(x_n) \le - \alpha \theta + (1 - \alpha) b_n + \alpha \varepsilon_n(\alpha)
$$
for all $n \ge n_1$ and $\alpha \in [0, \alpha_1]$. Recall that the codifferential mapping $D f$ uniformly
approximates the function $f$ on the set $\{ x_n \}_{n \in \mathbb{N}}$. Therefore there exists $\alpha_2 > 0$ such that
for all $n \in \mathbb{N}$ and $\alpha \in [0, \alpha_2]$ one has $|\varepsilon_n(\alpha)| < \theta / 3$. Hence for any
$n \ge n_1$ one has
$$
  f(x_n + \gamma h_n) - f(x_n) 
  \le - \gamma \frac{2 \theta}{3} + (1 - \gamma) b_n, \quad
  \gamma = \min\{ \alpha_0, \alpha_1, \alpha_2 \}.
$$
By our assumption $b_n \to 0$ as $n \to \infty$. Therefore for any sufficiently large $n \in \mathbb{N}$ one has
$$
  f(x_n + \gamma h_n) - f(x_n) \le - \gamma \frac{\theta}{3},
$$
which with the use of \eqref{DescentCond2} implies that 
$f(x_{n + 1}) \le f(x_n) - \gamma \theta / 3 + \varepsilon_n$ for all $n$ large enough. Consequently, taking
into account the fact that $\varepsilon_n \to 0$ as $n \to \infty$ one obtains that $f(x_n) \to - \infty$ 
as $n \to \infty$.

Thus, it remains to verify that inequality \eqref{MCDTechResult_WorseUCA} holds true. Denote 
$$
  g_n(\alpha) = \max_{(a, v) \in \underline{d} f(x_n) + z_n} \big( a + \alpha \langle v, h_n \rangle \big).
$$
Let us prove that there exist $n_1 \in \mathbb{N}$ and $\alpha_1 > 0$ such that $g_n(\alpha) = \eta_n(\alpha)$ for all
$n \ge n_1$ and $\alpha \in [0, \alpha_1]$. Then \eqref{MCDTechResult_WorseUCA} follows directly from
\eqref{MCDTechResult_UCA}.

From the inclusion $\underline{d}_{\nu_n} f(x_n) \subseteq \underline{d} f(x_n)$ it follows that
$g_n(\alpha) \ge \eta_n(\alpha)$ for all $\alpha \ge 0$ and $n \in \mathbb{N}$. We need to check that the converse
inequality holds true for any sufficiently small $\alpha \ge 0$, and for all $n \in \mathbb{N}$ large enough. It is
clear that 
\begin{equation} \label{SupOverClosure_vs_SupOverSet}
  \max_{(a, v) \in \underline{d} f(x_n)} \big( a + \alpha \langle v, h_n \rangle \big)
   = \sup_{(a, v) \in \extr \underline{d} f(x_n)} \big( a + \alpha \langle v, h_n \rangle \big).
\end{equation}
By our assumption the sets $\underline{d} f(x_n)$ and $\overline{d} f(x_n)$ lie within a bounded set $K$. 
Define $C_1 = \sup_{z \in K} \| z \|$, $C_2 = \sup_{n \in \mathbb{N}} \| h_n \|$ and $\alpha_1 = \nu^* / 4 C_1 C_2$.
Then for any $n \in \mathbb{N}$, $\alpha \in [0, \alpha_1]$ and $(a, v) \in \extr \underline{d} f(x_n)$ one has
$\alpha \langle v, h_n \rangle \ge -\nu^* / 4$, which implies that
$$
  a + \alpha \langle v, h_n \rangle 
  \begin{cases}
  \ge - 0.5 \, \nu^*, & \text{if~} a \ge - 0.25 \, \nu^*, \\
  < - 0.5 \, \nu^*, & \text{if~} a < -0.75 \, \nu^*.
  \end{cases}
$$
Therefore, for any $n \in \mathbb{N}$ and $\alpha \in [0, \alpha_1]$ one has
$$
  \sup_{(a, v) \in \extr \underline{d} f(x_n)} \big( a + \alpha \langle v, h_n \rangle \big) 
  = \sup_{(a, v) \in \extr \underline{d} f(x_n) \colon a \ge - 0.75 \nu^*} 
  \big( a + \alpha \langle v, h_n \rangle \big) \ge - \frac{\nu^*}{2}.
$$
By the definition of $\nu^*$ there exists $n_1 \in \mathbb{N}$ such that $\nu_n \ge 0.75 \nu^*$ for all $n \ge n_1$.
Consequently, taking into account the definition of $\underline{d}_{\nu} f(x)$ and equality
\eqref{SupOverClosure_vs_SupOverSet} one obtains that $g_n(\alpha) = \eta_n(\alpha)$ for all 
$\alpha \in [0, \alpha_1]$ and $n \ge n_1$, and the proof is complete.	 
\end{proof}

\subsection{Global convergence}

Now we can prove the global convergence of the MCD.

\begin{theorem} \label{Th_GlobalConvergence}
Suppose that $x^* \in U$ is a cluster point of a sequence $\{ x_n \}$ generated by the MCD, the codifferential mapping
$D f$ is continuous at $x^*$, and uniformly approximates the function $f$ in a neighbourhood of this point. Let also $f$
be bounded below on $U$, $\liminf_{n \to \infty} \nu_n > 0$ and $\liminf_{n \to \infty} \mu_n > 0$. Suppose finally that
one of the two following assumptions holds true:
\begin{enumerate}
\item{$\mathcal{H}$ is finite dimensional;
}
\item{$\{ (b, w) \in \overline{d} f(x_n) \mid b \le \widehat{\mu} \} \subseteq \overline{d}_{\mu_n} f(x_n)$ for some
$\widehat{\mu} > 0$, and for all sufficiently large $n \in \mathbb{N}$. \label{Assumpt_InfiniteDim_Hyperdiff}
}
\end{enumerate}
Then $x^*$ is an inf-stationary point of the function $f$. If, in addition, $f$ is convex, then $x^*$ is a point of
global minimum of the function $f$.
\end{theorem}

\begin{proof}
Arguing by reductio ad absurdum, suppose that $x^*$ is not an inf-stationary point of the function $f$. Then there
exists $z^* = (0, w^*) \in \extr \overline{d} f(x^*)$ such that 
$\theta = \min \{ \| (a, v) \|^2 \mid (a, v) \in \underline{d} f(x^*) + z^* \} > 0$. 
Our aim is to apply Lemma~\ref{Lemma_GlobalDivergence}.

Since $x^*$ is a cluster point of the sequence $\{ x_n \}$, there exists a subsequence $\{ x_{n_k} \}$ converging
to $x^*$. Therefore $\overline{d} f(x_{n_k}) \to \overline{d} f(x^*)$ as $k \to \infty$ in the Hausdorff metric due to
the continuity of the codifferential mapping $D f$ at $x^*$. Hence applying Lemma~\ref{Lemma_LimitOfExtremePoints} and
the fact that $\liminf_{n \to \infty} \mu_n > 0$ in the case when $\mathcal{H}$ is finite dimensional or
assumption \ref{Assumpt_InfiniteDim_Hyperdiff} in the case when $\mathcal{H}$ is infinite dimensional one obtains that
there exists a subsequence of the sequence $\{ x_{n_k} \}$, which we denote again by $\{ x_{n_k} \}$, and there exists
$z_{n_k} = (b_{n_k}, w_{n_k}) \in \overline{d}_{\mu_{n_k}} f(x_{n_k})$ such that $z_{n_k} \to z^*$, i.e.
$b_{n_k} \to 0$ and $w_{n_k} \to w^*$, as $k \to \infty$. Consequently, taking into account the fact that the
multifunction $\underline{d} f(\cdot)$ is Hausdorff continuous at $x^*$ by our assumption one obtains
that $\| (a_{n_k}(z_{n_k}), v_{n_k}(z_{n_k})) \|^2 > \theta / 2$ for all $k$ greater than some $k_1 \in \mathbb{N}$.

Denote $a_k = a_{n_k}(z_{n_k})$ and $v_k = v_{n_k}(z_{n_k})$. From the definition of $(a_k, v_k)$, and the necessary and
sufficient conditions for a minimum of a convex function on a convex set it follows that
\begin{equation} \label{SeparationTheorem_GeneralCase}
  - (a + b_{n_k}) a_k - \big\langle v + w_{n_k}, v_k \big\rangle
  \le - \big\| ( a_k, v_k ) \big\|^2 < - \frac{\theta}{2}
\end{equation}
for all $(a, v) \in \underline{d}_{\nu_{n_k}} f(x_{n_k})$ and $k \ge k_1$. Therefore
$$
  \sup_{(a, v) \in \underline{d}_{\nu_{n_k}} f(x_{n_k}) + z_{n_k}} ( a + \alpha \langle v, - v_k \rangle ) 
  \le - \alpha \frac{\theta}{2} 
  + \max_{(a, v) \in \underline{d}_{\nu_{n_k}} f(x_{n_k})} \big( (a + b_{n_k}) (1 + \alpha a_k) \big).
$$
for all $\alpha \ge 0$ and $k \ge k_1$. Taking into account the facts that the codifferential mapping $D f$ is
continuous at $x^*$, and $x_{n_k} \to x^*$ as $k \to \infty$ one obtains that the sequences 
$\{ \underline{d} f(x_{n_k}) \}$ and $\{ \overline{d} f(x_{n_k}) \}$ lie within a bounded set,
which, in particular, implies that the sequences $\{ a_k \}$ and $\{ v_k \}$ are bounded. Consequently, there exists
$\gamma > 0$ such that $-1 < \gamma a_k < 1$ for all $k \ge k_1$, which implies that
$$
  \sup_{(a, v) \in \underline{d}_{\nu_{n_k}} f(x_{n_k}) + z_{n_k}} ( a + \langle v, - \gamma v_k \rangle )
  \le - \gamma \frac{\theta}{2} + (1 + \gamma a_k) b_{n_k} 
  \qquad \forall k \ge k_1
$$
(here we used equalities \eqref{BothApproxZeroAtZero}). By definition $b_{n_k} \to 0$ as $k \to \infty$. Therefore
there exists $k_2 \ge k_1$ such that $b_{n_k} < \gamma \theta / 4$ for all $k \ge k_2$.

Observe that since $x_{n_k} \to x^*$ as $k \to \infty$ and $x^* \in U$, there exist $k_3 \in \mathbb{N}$ and $r > 0$
such that $B(x_{n_k}, r) \subset U$ for all $k \ge k_3$. Furthermore, from the fact that the codifferential mapping 
$D f$ uniformly approximates the function $f$ in a neighbourhood of $x^*$ it follows that $D f$ uniformly approximates
$f$ on the set $\{ x_{n_k} \}_{k \ge k_4}$ for some $k_4 \in \mathbb{N}$. Therefore applying
Lemma~\ref{Lemma_GlobalDivergence} with $\alpha_0 = \alpha_*$ (recall that $\alpha_*$ is the upper bound on the
step size in the MCD), the sequence $\{ x_n \}$ defined as $\{ x_{n_k} \}_{k \ge m}$, and the sequence $\{ h_n \}$
defined as $\{ - \gamma v_k \}_{k \ge m}$, where $m = \max\{ k_2, k_3, k_4 \}$, one obtains that 
$f(x_{n_k}) \to - \infty$ as $k \to \infty$, which is impossible due to the fact that $f$ is bounded below on $U$ (note
that the validity of inequality \eqref{DescentCond2} in this case follows directly from the definitions of $\alpha_n(z)$
and $x_{n + 1}$ in the MCD, and the fact that $f(x_{n + 1}) \le f(x_n)$ for all $n \in \mathbb{N}$).

It remains to note that if $f$ is convex, and $x^*$ is an inf-stationary point of $f$, then $f'(x^*, \cdot) \ge 0$ by
Proposition~\ref{Prp_NessOptCond}, which implies that $x^*$ is a point of global minimum of $f$ due to the convexity
assumption.	 
\end{proof}

\begin{remark}
For the validity of the theorem above it is sufficient to suppose that 
$\{ 0 \} \times \extr \overline{\partial} f(x^*) \subseteq \limsup_{k \to \infty} \extr \overline{d} f(x_{n_k})$, where
$x_{n_k} \to x^*$ as $k \to \infty$. Lemma~\ref{Lemma_LimitOfExtremePoints} guarantees that this inclusion always holds
true in the finite dimensional case. In order to ensure the validity of this inclusion in the infinite dimensional case
we utilized assumption \ref{Assumpt_InfiniteDim_Hyperdiff}, but it should be noted that in many applications the
validity of this inclusion can be verified directly.
\end{remark}

\begin{corollary} \label{Crlr_GlobalConvergence}
Let $\liminf_{n \to \infty} \nu_n > 0$, $\liminf_{n \to \infty} \mu_n > 0$, and $f$ be bounded below on $U$. Suppose
also that the codifferential mapping $D f$ is continuous on $U$ and locally uniformly approximates the function $f$ on
the set $U$. Let finally one of the two following assumptions be valid:
\begin{enumerate}
\item{$\mathcal{H}$ is finite dimensional;
}
\item{$\{ (b, w) \in \overline{d} f(x_n) \mid b \le \widehat{\mu} \} \subseteq \overline{d}_{\mu_n} f(x_n)$ for some
$\widehat{\mu} > 0$, and for all sufficiently large $n \in \mathbb{N}$.
}
\end{enumerate}
Then any cluster point $x^* \in U$ of the sequence $\{ x_n \}$ generated by the MCD is an inf-stationary point of 
the function $f$.
\end{corollary}

It is well-known that under some natural assumptions for any gradient method one has 
$\lim_{n \to \infty} \| \nabla f(x_n) \| = 0$, where $\{ x_n \}$ is a sequence generated by this method 
(see, e.g., \cite{Penot}, Theorem~2.5). Let us extend this result to the case of the MCD.
To this end, for any $\nu \ge 0$ introduce the function
$$
  \omega(x, \nu) = 
  \sup_{(0, w) \in \overline{d} f(x)} 
  \min_{u \in \cl \co \underline{d}_{\nu} f(x) + (0, w)} \| u \|^2
$$
that, in a sense, measures how far a point $x$ is from being an inf-stationary point of the function $f$. In particular,
$x$ is an inf-stationary point of $f$ iff $\omega(x, \nu) = 0$ for some $\nu \ge 0$. It should be noted that 
the function $\omega(x, \nu)$ is not continuous (or even l.s.c.) in $x$ in the general case, even if the function $f$ is
continuously codifferentiable. But $\omega$ is continuous in $x$, if $f$ is continuously hypodifferentiable, and 
$\nu = + \infty$, since in this case $\omega(x, + \infty) = \dist(0, \underline{d} f(x))$.

\begin{theorem} \label{Th_InfStatMeasure_Convergence}
Let $f$ be bounded below on $U$, a sequence $\{ x_n \}$ be generated by the MCD, and the sequences 
$\{ \underline{d} f(x_n) \}$ and $\{ \overline{d} f(x_n) \}$ be bounded. Suppose also that 
$\liminf_{n \to \infty} \nu_n = \nu^* > 0$, there exists $r > 0$ such that $B(x_n, r) \subset U$ for all 
$n \in \mathbb{N}$, and the codifferential mapping $D f$ uniformly approximates 
the function $f$ on the set $\{ x_n \}_{n \ge m}$ for some  $m \in \mathbb{N}$ (in particular, one can suppose that 
$D f$ uniformly approximates $f$ on the set $\{ x \in U \mid f(x) \le f(x_0) \}$). 
Then $\omega(x_n, \nu_n) \to  0$ as $n \to \infty$. In particular, if $f$ is
hypodifferentiable, i.e. $D f(\cdot) = [ \underline{d} f(\cdot), \{ 0 \} ]$, and $\nu_n \equiv + \infty$, 
then $\dist(0, \underline{d} f(x_n)) \to 0$ as $n \to \infty$.
\end{theorem}

\begin{proof}
Arguing by reductio ad absurdum, suppose that the theorem is false. Then there exist a subsequence $\{ x_{n_k} \}$ and
$\theta > 0$ such that $\omega(x_{n_k}, \nu_{n_k}) > \theta$ for all $k \in \mathbb{N}$. By the definition of
$\omega$ for any $k \in \mathbb{N}$ there exists $z_{n_k} = (0, w_{n_k}) \in \overline{d}_{\mu_{n_k}} f(x_{n_k})$ such
that $\| ( a_{n_k}(z_{n_k}), v_{n_k}(z_{n_k}) ) \|^2 \ge \theta$. Then taking into account the fact that the sequences 
$\{ \underline{d} f(x_n) \}$ and $\{ \overline{d} f(x_n) \}$ are bounded, and arguing in the same way as
in the proof of Theorem~\ref{Th_GlobalConvergence} one can easily check that there exists $\gamma > 0$ such that
$$
  \sup_{(a, v) \in \underline{d}_{\nu_{n_k}} f(x_{n_k}) + z_{n_k}} ( a + \langle v, - \gamma v_{n_k}(z_{n_k}) \rangle )
  \le - \gamma \theta 
  \qquad \forall k \in \mathbb{N}.
$$
Consequently, applying Lemma~\ref{Lemma_GlobalDivergence} one obtains that $f(x_{n_k}) \to - \infty$ as $k \to \infty$,
which contradicts the assumption that $f$ is bounded below on $U$.	 
\end{proof}

\begin{remark}
It is easy to verify that Theorems~\ref{Th_GlobalConvergence} and \ref{Th_InfStatMeasure_Convergence} remain to
hold true in the case when the search directions $v_n(z_n)$ are replaced by some approximations $\widetilde{v}_n(z_n)$
such that
\begin{equation} \label{GoodApproximation}
  \big\| \widetilde{v}_n(z_n) -  v_n(z_n) \big\| \le \varepsilon_n,
\end{equation}
where $\varepsilon_n \to + 0$ as $n \to \infty$. Namely, note that in this case one has
$$
  - (a + b_{n_k}) a_{n_k}(z_{n_k}) - \big\langle v + w_{n_k}, \widetilde{v}_{n_k}(z_{n_k}) \big\rangle
  \le - \frac{\theta}{2} + \varepsilon_{n_k} C \\
$$
for a sufficiently large $C > 0$, and for all $(a, v) \in \underline{d}_{\nu_{n_k}} f(x_{n_k})$ and $k \ge k_1$
(see~\eqref{SeparationTheorem_GeneralCase}). With the use of this estimate and the fact that $\varepsilon_{n_k} \to 0$
one can easily obtain the required results. In particular, one can extend Theorems~\ref{Th_GlobalConvergence} and
\ref{Th_InfStatMeasure_Convergence} to the case when instead of the sets $\underline{d} f(x_n)$ and $\overline{d}
f(x_n)$ one uses their approximations, provided these approximations are ``good enough'', i.e. provided inequality
\eqref{GoodApproximation} holds true for the corresponding approximate search directions. 

Similarly, Theorems~\ref{Th_GlobalConvergence} and \ref{Th_InfStatMeasure_Convergence} remain to hold true if
the step sizes $\alpha_n(z)$ are computed only approximately. Namely, it is sufficient to suppose that
\begin{multline*}
  f(x_n - \alpha_n(z) v_n(z)) \\
  \le \inf\Big\{ f(x_n - \alpha_n(z) v_n(z) \mid 
  \alpha \in [0, \alpha_*] \colon x_n - \alpha_n(z) v_n(z) \in U \Big\} + \varepsilon_n,
\end{multline*}
where $\varepsilon_n \to +0$ as $n \to \infty$. Thus, one can say that the MCD is somewhat robust with respect to
computational errors.
\end{remark}

Note that in the general case for some $z = (b, w) \in \overline{d}_{\mu_n} f(x_n)$ with $b > 0$ the corresponding
search direction $- v_n(z)$ might not be a descent direction of the function $f$, i.e. $f'(x_n, - v_n(z)) > 0$. 

\begin{example}
Let $\mathcal{H} = \mathbb{R}^2$, $x_0 = (0, 0)$ and
$$
  f(x_1, x_2) = \max\{ x_1 + x_2, x_1^2 + x_2^2 - 1 \}
  + \min\{ x_1^3 + x_2^3, - 2 x_1 - x_2 + 1, - x_1 - 2 x_2 + 2 \},
$$
With the use of Therorem~\ref{Theorem_Codiff_Max_Min}, Corollary~\ref{Corollary_Codiff_Lin_Combin} and
Example~\ref{Example_Differentiable} one gets
$$	
  \underline{d} f(x_0) = 
  \co\left\{ \begin{pmatrix} 0 \\ 1 \\ 1 \end{pmatrix}, \begin{pmatrix} -1 \\ 0 \\ 0 \end{pmatrix} \right\}, \quad
  \overline{d} f(x_0) = \co\left\{ \begin{pmatrix} 0 \\ 0 \\ 0 \end{pmatrix},
  \begin{pmatrix} 1 \\ -2 \\ -1 \end{pmatrix}, \begin{pmatrix} 2 \\ -1 \\ -2 \end{pmatrix} \right\}
$$
(note that equalities \eqref{BothApproxZeroAtZero} hold true). Let $\nu = 0.5$ and $\mu = 1$, and define
$$
  \underline{d}_{\nu} f(x_0) = \left\{ \begin{pmatrix} 0 \\ 1 \\ 1 \end{pmatrix} \right\}, \quad
  \overline{d}_{\mu} f(x_0) = \left\{ \begin{pmatrix} 0 \\ 0 \\ 0 \end{pmatrix},
  \begin{pmatrix} 1 \\ -2 \\ -1 \end{pmatrix} \right\}.
$$
Observe that $f$ is differentiable at $x_0$ and $\nabla f(x_0) = (1, 1)$. Hence for 
$z = (1, -2, -1) \in \overline{d}_{\mu} f(x_0)$ one has $v_0(z) = (- 1, 0)$, and $f'(x_0, -v_0(z)) = 1$.
\end{example}

Thus, the only reasonable step size rule in the case when $z = (b, w) \in \overline{d}_{\mu_n} f(x_n)$ is such that
$b > 0$ is the minimization along the direction $- v_n(z)$ over some interval $[0, \alpha_*]$. It should be noted that
such directions $v_n(z)$ cannot be excluded, since otherwise the MCD might converge to a non-stationary point (see
the proof of Theorem~\ref{Th_GlobalConvergence}). Let us also note that the choice of $\alpha_*$ is usually
heuristic, since if $\alpha_*$ is too small, then one might not benifit from the use of a non-decent direction $v_n(z)$
(i.e. the method would be unable to ``jump over'' a local minimum), while if $\alpha_*$ is too large, then the line
search might take unreasonably long time. 

Note finally that Lemma~\ref{Lemma_DescentDirections} implies that for any 
$z = (0, w) \in \overline{d}_{\mu_n} f(x_n)$ one has $f'(x_n, -v_n(z)) \le - \| (a_n(z), v_n(z)) \|^2$. In this case one
can compute the step size $\alpha_n(z)$ as follows:
\begin{equation} \label{ArmijoStepSize}
  \alpha_n(z) = \max_{k \in \mathbb{N} \cup \{ 0 \} } \Big\{ \gamma^k \Bigm|
  f(x_n - \gamma^k v_n(z)) - f(x_n) \le - \sigma \gamma^k \big\| (a_n(z), v_n(z)) \big\|^2 \Big\}.
\end{equation}
Here $\sigma, \gamma \in (0, 1)$ are fixed. It is easy to see that if $D f$ uniformly approximates the function
$f$ in a neighbourhood of a point $x^*$, and a sequence $\{ x_n \}$ generated by the MCD converges to $x^*$, then for
any $x_n$ in a neighbourhood of $x^*$ step sizes \eqref{ArmijoStepSize} are bounded away from zero. With the use of
this result one can verify that Theorems~\ref{Th_GlobalConvergence} and \ref{Th_InfStatMeasure_Convergence} remain to
hold true if in the MCD for any $z = (0, w) \in \overline{d}_{\mu_n} f(x_n)$ one uses the step size
rule~\eqref{ArmijoStepSize}.

\section{The quadratic regularization of the MCD}
\label{Section_QR_MCD}

In this section we present and analyse a different method for minimizing a codifferentiable function $f$, which we call 
\textit{the quadratic regularization of the MCD}. In the case when $f$ is the maximum of a finite family of continuously
differentiable functions, this method coincides with the PPP algorithm
\cite{Pschenichny,PironneauPolak,DaugavenMalozemov,Polak}. Since the quadratic regularization of the MCD can be easily
applied not only to the problem of unconstrained minimization of a codifferentiable function, but also to the problem of
minimizing a codifferentiable function over a convex set, below we consider the constrained version of the problem.

Let, as above, $U \subseteq \mathcal{H}$ be an open set, a function $f \colon U \to \mathbb{R}$ be codifferentiable on
$U$, and $D f$ be its fixed codifferential mapping. Let also $A \subset U$ be a closed convex set. Below, we study the
problem of minimizing the function $f$ over the set $A$. Let us first obtain necessary optimality conditions for this
problem.

For any $x \in U$ and $z \in \overline{d} f(x)$ introduce the continuous convex function
\begin{equation} \label{QuadReg_Def}
  \varphi(h, z, x, \nu) = 
  \max_{(a, v) \in \cl\co\underline{d}_{\nu} f(x) + z} \big( a + \langle v, h \rangle \big)
  + \frac{1}{2} \| h \|^2 \quad \forall h \in \mathcal{H}.
\end{equation}
In the case when some sequences $\{ \nu_n \} \subset [0, + \infty]$ and $\{ x_n \} \subset U$ are given,
we denote $\varphi_n(h, z) = \varphi(h, z, x_n, \nu_n)$. Let us note that the quadratic term is introduced into the
definition of the function $\varphi(h, z, x, \nu)$ in order to ensure that this function attains a global minimum in $h$
on the set $A - x$.

\begin{proposition} \label{Prp_OptCond_ConvexSet}
Let $x^* \in A$ be a point of local minimum of the function $f$ on the set $A$. Then for any $\nu \ge 0$ one has
\begin{equation} \label{OptCond_ConvexSet}
  \{ 0 \} = \argmin_{h \in A - x^*} \varphi (h, z, x^*. \nu)
  \qquad \forall z = (0, w) \in \extr \overline{d} f(x^*).
\end{equation}
Furthermore, \eqref{OptCond_ConvexSet} holds true iff one of the two following statements is valid:
\begin{enumerate}
\item{$0$ is a global minimizer of the function $\varphi(\cdot, z, x^*, \nu) - \| \cdot \|^2 / 2$ on 
the set $A - x^*$; \label{EquivOptCond_ConvexSet}
}

\item{$f'(x^*, h) \ge 0$ for all $h \in A - x^*$.
}
\end{enumerate}
In particular, optimality condition \eqref{OptCond_ConvexSet} is independent of the choice of a codifferential and
parameter $\nu \ge 0$.
\end{proposition}

\begin{proof}
Let $z = (0, w) \in \extr \overline{d} f(x^*)$ be arbitrary. Applying \cite{Dolgopolik_CalcVar}, Theorem~2.8 (see also
\cite{Dolgopolik_NonhomConvApprox}, Theorem~5) one obtains that $0$ is a point of global minimum of the function
$\varphi(\cdot, z, x^*, \nu) - \| \cdot \|^2 / 2$ on the set $A - x^*$. Hence taking into account the fact that 
the subdifferential of this convex function at the origin coincides with the subdifferential of the function
$\varphi(\cdot, z, x^*, \nu)$ at the origin one obtains that \eqref{OptCond_ConvexSet} holds true. Furthermore, from
the coincidence of the subdifferentials at the origin it follows that \eqref{OptCond_ConvexSet} holds true iff
condition~\ref{EquivOptCond_ConvexSet} holds true.

Applying the theorem about the subdifferential of the supremum of a family of convex functions (see, e.g.,
\cite[Thm.~4.2.3]{IoffeTihomirov}), the definition of $\underline{d}_{\nu} f(x)$ and the first equality in
\eqref{BothApproxZeroAtZero} one obtains that $\partial_h \varphi(0, z, x^*, \nu) = \underline{\partial} f(x^*) + w$, 
where $\partial_h \varphi(0, z, x^*, \nu)$ is the subdifferential (in the sense of convex analysis) of the function
$\varphi(\cdot, z, x^*, \nu)$ at the origin. Hence with the use of the standard necessary and sufficient condition
for a minimum of a convex function on a convex set one obtains that
$$
  \max_{v \in \underline{\partial} f(x^*) + w} \langle v, h \rangle \ge 0
  \quad \forall h \in A - x^* \quad \forall (0, w) \in \extr \overline{d} f(x^*).
$$
Taking the infimum over all $w \in \extr \overline{\partial} f(x^*)$ (clearly,
$\extr \overline{\partial} f(x^*) = \{ w \in \mathcal{H} \mid (0, w) \in \extr \overline{d} f(x^*) \}$), and applying
the Krein-Milman theorem one finds that 
\begin{equation} \label{OptCond_ConvexSet_DirectDerivative}
  f'(x^*, h) = \max_{v \in \underline{\partial} f(x^*)} \langle v, h \rangle + 
  \min_{w \in \overline{\partial} f(x^*)} \langle w, h \rangle \ge 0 \quad \forall h \in A - x^*.
\end{equation}
Arguing backwards one can check that if \eqref{OptCond_ConvexSet_DirectDerivative} is valid, then optimality condition
\eqref{OptCond_ConvexSet} holds true as well. 	 
\end{proof}

A point $x^* \in A$ satisfying optimality condition
$$
  \{ 0 \} = \argmin_{h \in A - x^*} \varphi (h, z, x^*, \nu)
  \qquad \forall z = (0, w) \in \extr \overline{d} f(x^*).
$$
for some $\nu \ge 0$ is called an inf-stationary point of the function $f$ on the set $A$. With the use of this
optimality condition, which is independent of the choice of a codifferential and $\nu \ge 0$, we can design the
quadratic regularization of the MCD (QR-MCD). The scheme of this method is as follows.

\begin{enumerate}
\item{Choose sequences $\{ \nu_n \}, \{ \mu_n \} \subset [0, + \infty]$, the upper bound $\alpha_* \in (0, + \infty]$
on the step size, and an initial point $x_0 \in A$.
}
\item{$n$th iteration ($n \ge 0$).
  \begin{enumerate}
    \item{Compute $\underline{d}_{\nu_n} f(x_n)$ and $\overline{d}_{\mu_n} f(x_n)$.
    }
    \item{For any $z = (b, w) \in \overline{d}_{\mu_n} f(x_n)$ compute
    $$
      \{ h_n(z) \} = \argmin\Big\{ \varphi_n(h, z) \Bigm| h \in A - x_n  \Big\}
    $$
    }
    \item{For any $z \in \overline{d}_{\mu_n} f(x_n)$ compute
    $$
      \alpha_n(z) \in \argmin\Big\{ f(x_n + \alpha h_n(z)) \Bigm| 
      \alpha \in [0, \alpha_*] \colon x_n + \alpha h_n(z) \in A \Big\}.
    $$
    }
    \item{Compute
    $$
      z_n \in \argmin\Big\{ f(x_n + \alpha_n(z) h_n(z)) \Bigm| z \in \overline{d}_{\mu_n} f(x_n) \Big\},
    $$
    and define $x_{n + 1} = x_n + \alpha_n(z_n) h_n(z_n)$.
    }
  \end{enumerate}
}
\end{enumerate}
Hereinafter, we suppose that the step sizes $\alpha_n(z)$, and the vectors $z_n$ are correctly defined.

\begin{remark}
Observe that the function $\varphi_n(\cdot, z)$ is strictly convex, continuous (since it is obviously bounded on bounded
sets), and $\varphi_n(h, z) \to + \infty$ as $\| h \| \to + \infty$. Therefore taking into account the facts that
$\mathcal{H}$ is a Hilbert space, and the convex set $A$ is closed, one obtains that the search directions $h_n(z)$ are
well-defined. Furthermore, note that $x_n$ is an inf-stationary point of the function $f$ on the set $A$ iff
$h_n(z) = 0$ for all $z = (0, w) \in \overline{d}_{\mu_n} f(x_n)$.
\end{remark}

Let us first extend Lemma~\ref{Lemma_DescentDirections} to the the case of the QR-MCD.

\begin{lemma}
Let a sequence $\{ x_n \}$ be generated by the QR-MCD. Then for any $n \in \mathbb{N}$ and 
$z = (0, w) \in \overline{d}_{\mu_n} f(x_n)$ such that $h_n(z) \ne 0$ one has $f'(x_n, h_n(z)) \le - \| h_n(z) \|^2$. In
particular, if $x_n$ is not an inf-stationary point of the function $f$ on the set $A$, then $f(x_{n + 1}) < f(x_n)$.
\end{lemma}

\begin{proof}
Fix an arbitrary $z = (0, w) \in \overline{d}_{\mu_n} f(x_n)$ such that $h_n(z) \ne 0$. Equalities
\eqref{BothApproxZeroAtZero} imply that $\varphi_n(0, z) = 0$. Therefore $\varphi_n(h_n(z), z) \le 0$ or, equivalently,
$$
  \max_{(a, v) \in \cl\co\underline{d}_{\nu_n} f(x_n) + z} \big( a + \langle v, h_n(z) \rangle \big) \le
  - \frac{1}{2} \| h_n(z) \|^2.
$$
Recall that by definition 
$\{ 0 \} \times \extr \underline{\partial} f(x_n) \subseteq \underline{d}_{\nu_n} f(x_n)$. Consequently, applying the
inequality above and the definition of quasidifferential one obtains that
$$
  f'(x_n, h_n(z)) \le \max_{v \in \underline{\partial} f(x_n) + w} \langle v, h_n(z) \rangle
  \le - \frac{1}{2} \| h_n(z) \|^2,
$$
which completes the proof.	 
\end{proof}

Now we can prove the global convergence of the QR-MCD.

\begin{theorem} \label{Th_QuadReg_GlobalConvergence}
Suppose that $x^*$ is a cluster point of a sequence $\{ x_n \} \subset A$ generated by the QR-MCD, the codifferential
mapping $D f$ is continuous at $x^*$, and uniformly approximates the function $f$ in a neighbourhood of this point. Let
also $f$ be bounded below on $A$, $\liminf_{n \to \infty} \nu_n > 0$ and $\liminf_{n \to \infty} \mu_n > 0$. Suppose
finally that one of the two following assumptions holds true:
\begin{enumerate}
\item{$\mathcal{H}$ is finite dimensional;
}
\item{$\{ (b, w) \in \overline{d} f(x_n) \mid b \le \widehat{\mu} \} \subseteq \overline{d}_{\mu_n} f(x_n)$ for some
$\widehat{\mu} > 0$ and for all sufficiently large $n \in \mathbb{N}$.
}
\end{enumerate}
Then $x^*$ is an inf-stationary point of the function $f$ on the set $A$. If, in addition, $f$ is convex, then $x^*$ is
a point of global minimum of $f$ on $A$.
\end{theorem}

\begin{proof}
Arguing by reductio ad absurdum, suppose that $x^*$ is not an inf-stationary point of the function $f$ on the set $A$.
Then there exists $z^* = (0, w^*) \in \extr \overline{d} f(x^*)$ such that 
$\min_{h \in A - x^*} \varphi(h, z^*, x^*, + \infty) = - \theta < 0$. Denote by $h^*$ a point of global minimum of the
function $\varphi(\cdot, z^*, x^*, + \infty)$ on the set $A - x^*$. Our aim is to apply
Lemma~\ref{Lemma_GlobalDivergence}.

Arguing in the same way as in the proof of Theorem~\ref{Th_GlobalConvergence} one can check that there exist a
subsequence $\{ x_{n_k} \}$ and a sequence $z_{n_k} = (b_{n_k}, w_{n_k}) \in \overline{d}_{\mu_{n_k}} f(x_{n_k})$ such
that $x_{n_k} \to x^*$ and $z_{n_k} \to z^*$ as $k \to \infty$. From the continuity of the codifferential mapping $D f$
at $x^*$ it follows that $\underline{d} f(x_{n_k}) \to \underline{d} f(x^*)$ in the Hausdorff metric. Applying this fact
it is easy to deduce that there exists $k_0 \in \mathbb{N}$ such that
$$
  \big| \varphi(h^*, z_{n_k}, x_{n_k}, + \infty) - \varphi(h^*, z^*, x^*, + \infty) \big| < \frac{\theta}{2}
  \quad \forall k \ge k_0,
$$
which implies that
$$
  \varphi_{n_k}(h^*, z_{n_k}) \le \varphi(h^*, z_{n_k}, x_{n_k}, + \infty) < - \frac{\theta}{2}
  \quad \forall k \ge k_0,
$$
and, moreover, $\varphi_{n_k}(h_{n_k}(z_{n_k}), z_{n_k}) < - \theta / 2$ for all $k \ge k_0$. Therefore
\begin{equation} \label{QuadReg_HypodiffUpperEstimate}
  \sup_{(a, v) \in \underline{d}_{\nu_{n_k}} f(x_{n_k}) + z_{n_k}} \big( a + \langle v, h_{n_k}(z_{n_k} \rangle \big) 
  < \varphi_{n_k}(h_{n_k}(z_{n_k}), z_{n_k}) < - \frac{\theta}{2}.
\end{equation}
for all $k \ge k_0$.

Clearly, the sequences $\{ \underline{d} f(x_{n_k}) \}$ and $\{ \overline{d} f(x_{n_k}) \}$ are bounded due to the
continuity of the codifferential mapping $D f$ at $x^*$. Therefore there exist $c_1, c_2 \in \mathbb{R}$ such that
$\varphi_{n_k}(h, z_{n_k}) \ge 0.5 \| h \|^2 + c_1 \| h \| + c_2$ for all $h \in \mathcal{H}$ and $k \in \mathbb{N}$,
which implies that the sequence $\{ h_{n_k}(z_{n_k}) \}$ is bounded. Hence taking into account
inequality~\eqref{QuadReg_HypodiffUpperEstimate}, and the fact that $D f$ uniformly approximates $f$ in a neighbourhood
of $x^*$, and applying Lemma~\ref{Lemma_GlobalDivergence} one obtains that $f(x_{n_k}) \to - \infty$ as $k \to \infty$,
which is impossible by virtue of the fact that $f$ is bounded below on $A$.	 
\end{proof}

Similarly to the case of the MCD, introduce the function
$$
  \omega_2(x, \nu) = \sup_{z = (0, w) \in \overline{d} f(x)} \| h(z, x, \nu) \|^2,
$$
that measures how far a point $x$ is from being an inf-stationary point of the function $f$ on the set $A$. Here 
$h(z, x, \nu)$ is a point of global minimum of the function $\varphi(\cdot, z, x, \nu)$ on the set $A - x$. It
is easy to see that $x^*$ is an inf-stationary point of the function $f$ on the set $A$ iff $\omega_2(x^*, \nu) = 0$ for
some $\nu \ge 0$.

Arguing in a similar way to the proof of Theorem~\ref{Th_InfStatMeasure_Convergence} one can easily verify that the
following result holds true.

\begin{theorem} \label{Th_QuadReg_InfStatMeasure_Convergence}
Let $f$ be bounded below on $A$, a sequence $\{ x_n \} \subset A$ be generated by the QR-MCD, and the sequences
$\{ \underline{d} f(x_n) \}$ and $\{ \overline{d} f(x_n) \}$ be bounded. Suppose also that there exists
$r > 0$ such that $B(x_n, r) \subset U$ for all $n \in \mathbb{N}$, the codifferentiable mapping $D f$ uniformly
approximates the function $f$ on the set $\{ x_n \}_{n \ge m}$ for some $m \in \mathbb{N}$, and 
$\liminf_{n \to \infty} \nu_n > 0$. Then $\omega_2(x_n, \nu_n) \to  0$ as $n \to \infty$. 
\end{theorem}

\begin{remark}
{(i)~As in the case of the MCD, one can verify that Theorems~\ref{Th_QuadReg_GlobalConvergence} and
\ref{Th_QuadReg_InfStatMeasure_Convergence} remain to hold true if instead of the search directions $h_n(z)$ one uses
their approximations $\widetilde{h}_n(z)$ such that $\| \widetilde{h}_n(z) - h_n(z) \| < \varepsilon_n$, where
$\varepsilon_n \to +0$ as $n \to \infty$. In order to prove this result one needs to note that if the sequences 
$\{ \underline{d} f(x_n) \}$ and $\{ \overline{d} f(x_n) \}$ are bounded, then the corresponding functions
$\varphi_n(h, z)$ are uniformly (in $n$) bounded on any bounded set, which, as it is well-known from convex analysis,
implies that these functions are Lipschitz continuous on any bounded set with a Lipschitz constant independent of $n$.
Also, one can verify that Theorems~\ref{Th_QuadReg_GlobalConvergence} and \ref{Th_QuadReg_InfStatMeasure_Convergence}
remain valid, if one uses the step size rule similar to step size rule \eqref{ArmijoStepSize} 
for any $z = (0, w) \in \overline{d}_{\mu_n} f(x_n)$.
}

\noindent{(ii)~For any $x \in U$ and $z \in \overline{d} f(x)$ introduce the function
$$
  \psi(h, z, x, \nu) = \sup_{(a, v) \in \underline{d}_{\nu} f(x) + z} \big( a + \langle v, h \rangle \big).
$$
From Proposition~\ref{Prp_OptCond_ConvexSet} it follows that $x^*$ is an inf-stationary point of the function $f$ on
the set $A$ iff $\{ 0 \} \in \argmin_{h \in A - x^*} \psi(h, z, x, \nu)$ for all 
$z = (0, w) \in \extr \overline{d} f(x^*)$ or, equivalently, iff 
$\{ 0 \} \in \argmin_{h \in K \cap (A - x^*)} \psi(h, z, x, \nu)$ for all 
$z = (0, w) \in \extr \overline{d} f(x^*)$, where $K$ is a set containing a neighbourhood of zero. Bearing this result
in mind, one can propose a different method for minimizing a codifferentiable function on a convex set. Namely, instead
of minimizing the function $\varphi_n(h, z)$ one can compute a search direction as follows:
$$
  \{ h_n(z) \} \in \argmin\Big\{ \psi(h, z, x_n, \nu_n) \Bigm| h \in K \cap (A - x_n) \Big\}.
$$
Here $K \subset \mathcal{H}$ is a bounded closed convex set containing a neighbourhood of zero. It is natural to call
the modification of the QR-MCD utilizing these search directions \textit{the primal regularization of the MCD}. Let us
note that one can easily extend the results of this section to the case of the primal regularization of the MCD.
}
\end{remark}

\section{The rate of convergence: an example}
\label{Section_Rate_of_Convergence}

In this section we discuss the rate of convergence of the MCD. Note that if the function $f$ is smooth, then the MCD
coincides with the gradient descent. Therefore it is natural to expect that the MCD converges linearly, provided a
suitable second order sufficient optimality condition holds true at the limit point. A rigorous analysis of the rate of
convergence obviously requires the use of very cumbersome second order approximation of codifferentiable functions (such
as the so-called \textit{second order codifferentials}, see~\cite{DemRub_book}), and is complicated by the fact that in
every iteration of the method one uses multiple search directions some of which might not be descent directions.
Furthermore, the MCD (unlike the optimality conditions) is not invariant with respect to the choice of a codifferential,
and one can provide examples in which a poor choice of a codifferential mapping significantly slows down the method.
That is why we do not present a detailed analysis of the rate of convergence of the MCD here, and leave it as an open
problem for future research. Instead, we consider only a certain class of nonconvex codifferentiable functions for which
the rate of convergence can be easily estimated with the use of some existing results.

Let the function $f$ have the form
\begin{equation} \label{MaxPlusMinOfConvexFunc}
  f(x) = \sum_{k = 1}^m \max_{i \in I_k} g_{ki}(x) + \sum_{l = 1}^s \min_{j \in J_l} u_{lj}(x),
\end{equation}
where $g_{ki}, u_{lj} \colon \mathcal{H} \to \mathbb{R}$ are continuously differentiable functions, 
$I_k = \{ 1, \ldots, p_k \}$, $J_l = \{ 1, \ldots, q_l \}$. Particular examples of functions of the form
\eqref{MaxPlusMinOfConvexFunc}, including some functions arising in cluster analysis, can be found in
\cite{DemBagRub}.

Observe that $f$ is a nonconvex nonsmooth function (even if all functions $g_{ki}$ and $u_{lj}$ are convex).
Furthermore, $f$ is continuously codifferentiable, and applying Theorem~\ref{Theorem_Codiff_Max_Min} and
Corollary~\ref{Corollary_Codiff_Lin_Combin} one can define
\begin{multline*}
  D f(x) = \Bigg[ \sum_{k = 1}^m \co\big\{ (g_{ki}(x) - g_k(x), \nabla g_{ki}(x)) \bigm| i \in I_k \big\}, \\
  \sum_{l = 1}^s \co\big\{ (u_{lj}(x) - u_l(x), \nabla u_{lj}(x)) \bigm| j \in J_l \big\} \Bigg],
\end{multline*}
where $g_k(x) = \max_{i \in I_k} g_{ki}(x)$, and $u_l(x) = \min_{j \in J_l} u_{lj}(x)$. Additionally, one can easily
check that the codifferential mapping $D f(x)$ locally uniformly approximates the funciton $f$, provided the gradients
of the functions $g_{ki}$ and $u_{lj}$ are locally uniformly continuous (see~Theorem~\ref{Theorem_Codiff_Max_Min}).

Let us apply the QR-MCD to the function $f$. Fix arbitrary $\nu \ge 0$ and $\mu \ge 0$. For any $x \in \mathcal{H}$,
introduce the index sets 
\begin{gather*}
  I_{k, \nu}(x) = \{ i \in I_k \mid g_{ki}(x) \ge g_k(x) - \nu \}, \:
  J_{l, \mu}(x) = \{ j \in J_l \mid u_{lj}(x) \le u_l(x) + \mu \}, \\
  I_{\nu}(x) = I_{1, \nu}(x) \times \ldots \times I_{m, \nu}(x), \quad
  J_{\mu}(x) = I_{1, \mu}(x) \times \ldots \times J_{s, \mu}(x),
\end{gather*}
and denote $I(x) = I_0(x)$ and $J(x) = J_0(x)$. Define also
\begin{gather*}
  \underline{d}_{\nu} f(x) = 
  \sum_{k = 1}^m \big\{ (g_{ki}(x) - g_k(x), \nabla g_{ki}(x)) \bigm| i \in I_{k, \nu}(x) \big\}, \\
  \overline{d}_{\mu} f(x) = 
  \sum_{l = 1}^s \big\{ (u_{lj}(x) - u_l(x), \nabla u_{lj}(x) \bigm| j \in J_{l, \mu}(x) \big\}.
\end{gather*}
It is easy to see that inclusions \eqref{ReducedCodifferentialDef} hold true. Thus, we can minimize the function $f$
with the use of the QR-MCD with the sets $\underline{d}_{\nu} f(x)$ and $\overline{d}_{\mu} f(x)$ defined as above
(note that these sets are finite).

For any multi-index $\lambda = (j_1, \ldots, j_s) \in J_1 \times \ldots \times J_s$ introduce the function
$$
  f_{\lambda}(x) = \sum_{k = 1}^m \max_{i \in I_k} g_{ki}(x) + \sum_{l = 1}^s u_{l j_l}(x).
$$
By definition one has $f(x) = \min_{\lambda} f_{\lambda}(x)$. In essence, the QR-MCD for the function $f$ can be veiwed
as a methods that simultaneously minimizes the functions $f_{\lambda}$, $\lambda \in J_{\mu}(x)$, with the use of 
the PPP algorithm. This idea will allow us to estimate the rate of convergence of this method with the use of
some existing results on the convergence of the PPP algorithm (see \cite{Polak}, Sections~2.4.1--2.4.4).

Under the assumption that the functions $g_{ki}$ and $u_{lj}$ are strongly convex, we can prove that the QR-MCD
converges with a linear rate. In order to utilize the existing results on the convergence of the PPP algorithm, below we
suppose that $\nu = + \infty$, and the space $\mathcal{H}$ is finite dimensional. It should be noted that the following
theorem remains valid in the general case, but for the sake of shortness we do not present the proof of this result
here.

\begin{theorem} \label{Th_LinearConvergenceOfQuadRegMCD}
Let $\mathcal{H} = \mathbb{R}^d$, $\alpha_* \ge 1$, $\nu_n \equiv + \infty$, and $\lim_{n \to \infty} \mu_n > 0$. Let
also the functions $g_{ki}$ and $u_{lj}$ be twice continuously differentiable. Suppose that there exist $M > m > 0$ such
that
$$
  m \| y \|^2 \le \langle y, \nabla^2 g_{ki}(x) y \rangle \le M \| y \|^2, \quad
  m \| y \|^2 \le \langle y, \nabla^2 u_{lj}(x) y \rangle \le M \| y \|^2, \quad
$$
for any $x, y \in \mathbb{R}^d$, and for all $i \in I_k$, $k \in \{1, \ldots, m \}$ and $j \in J_l$, 
$l \in \{1, \ldots, s \}$. Suppose, finally, that a sequence $\{ x_n \}$ generated by the QR-MCD for the function $f$
converges to a point $x^*$. Then $x^*$ is a point of strict local minimum of the function $f$, and there exist 
$c \in (0, 1)$, $Q > 0$, and $n_0 \in \mathbb{N}$ such that
\begin{equation} \label{QuadRegMCD_LinearRateOfConvergence}
  \big[ f(x_{n + 1}) - f(x^*) \big] \le c \big[ f(x_n) - f(x^*) \big], \quad
  \big\| x_{n + 1} - x^* \big\| \le Q c^{n / 2},
\end{equation}
for any $n \ge n_0$, i.e. the QR-MCD converges linearly.
\end{theorem}

\begin{proof}
By Theorem~\ref{Th_QuadReg_GlobalConvergence} the point $x^*$ is an inf-stationary point of the function $f$.
Therefore, as it is easy to check, $0 \in \partial f_{\lambda}(x^*)$ for any $\lambda \in J(x^*)$, where 
$\partial f_{\lambda}(x^*)$ is the subdifferential of the convex function $f_{\lambda}$ at $x^*$ in the sense of convex
analysis. Consequently, $x^*$ is a point of global minimum of $f_{\lambda}$ for all $\lambda \in J(x^*)$. Observe that
every function $f_{\lambda}$ is strictly convex as the sum of the maximum of convex functions and a strongly convex
function. Hence $x^*$ is a point of \textit{strict} global minimum of $f_{\lambda}$ for all $\lambda \in J(x^*)$, which,
as it is easily seen, implies that $x^*$ is a point of strict local minimum of the function $f$.

Let us now estimate the rate of convergence. Since the functions $u_{lj}$ are continuous, there exists a
neighbourhood $U$ of $x^*$ such that $J(x) \subseteq J(x^*)$ for all $x \in U$. Hence taking into account the fact that
$x_n \to x^*$ as $n \to \infty$, one obtains that there exists $n_0 \in \mathbb{N}$ such that $x_n \in U$ for all 
$n \ge n_0$.

Fix arbitrary $\sigma, \gamma \in (0, 1)$, $n \ge n_0$ and $\lambda = (j_1, \ldots, j_s) \in J(x_n)$. Define
$z_n = ( 0, \sum_{l = 1}^s \nabla u_{l j_l}(x_n) )$. Then $z_n \in \overline{d}_{\mu} f(x_n)$. Define also
$$
  \{ \gamma_n \} = \argmax_{k \in \mathbb{N} \cup \{ 0 \}}\Big\{ \gamma^k \Bigm|
  f_{\lambda}(x_n + \gamma^k h_n(z)) - f_{\lambda}(x_n) \le \gamma^k \sigma \varphi_n(h_n(z_n), z_n) \Big\}.
$$
Let us check that $\gamma_n$ is correctly defined. Indeed, from the definitions of $f_{\lambda}$ and 
$\varphi_n(h, z)$ (see~\eqref{QuadReg_Def}) it follows that $f_{\lambda}'(x_n, h) = (\varphi_n(\cdot, z_n))'(0, h)$ for
all $h \in \mathcal{H}$. Taking into account the convexity of the function $\varphi_n(\cdot, z_n)$ and the fact that
$\varphi_n(0, z_n) = 0$ one obtains that $(\varphi_n(\cdot, z_n))'(0, h) \le \varphi_n(h, z_n)$ for all 
$h \in \mathcal{H}$, which yields that $f(x_n + \alpha h_n(z_n)) - f(x_n) \le \alpha \sigma \varphi_n(h_n(z_n), z_n)$
for any sufficiently small $\alpha$ (note that if $\varphi_n(h_n(z_n), z_n) = 0$, then $h_n(z_n) = 0$). Therefore
$\gamma_n \in (0, 1]$ is correctly defined.

Observe that $h_n(z_n)$ is the search direction, and $\gamma_n$ is the step size of the PPP algorithm for the function
$f_{\lambda}$ (see~\cite{Polak}, Algorithm~2.4.1). Consequently, applying ~\cite{Polak}, Theorem~2.4.5 one gets that
there exists $c \in (0, 1)$ depending only on $\gamma$, $\sigma$, $m$ and $M$ such that
$$
  \big[ f_{\lambda}(x_n + \gamma_n h_n(z_n)) - f_{\lambda}(x^*) \big] 
  \le c \big[ f_{\lambda}(x_n) - f_{\lambda}(x^*) \big].
$$
Recall that $\alpha_* \ge 1$. Therefore 
$$
  f(x_n + \alpha_n(z_n) h_n(z_n)) \le f(x_n + \gamma_n h_n(z_n)) \le f_{\lambda}(x_n + \gamma_n h_n(z_n)).
$$
Hence and from the fact that $\lambda \in J(x_n) \subseteq J(x^*)$ it follows that
$$
  \big[ f(x_n + \alpha_n(z_n) h_n(z_n)) - f(x^*) \big] \le c \big[ f(x_n) - f(x^*) \big].
$$
Consequently, taking into account the fact that $f(x_{n + 1}) \le f(x_n + \alpha_n(z_n) h_n(z_n))$ by definition one
gets that the first inequality in \eqref{QuadRegMCD_LinearRateOfConvergence} holds true.

Applying inequality (2.4.9h) from \cite{Polak} one obtains that 
$f_{\lambda}(x) - f_{\lambda}(x^*) \ge m \| x - x^* \|^2 / 2$ for all $x \in \mathbb{R}^d$ and 
$\lambda \in J(x^*)$. Consequently, $f(x) - f(x^*) \ge m \| x - x^* \|^2 / 2$ for any $x \in U$. Combining this
inequality with the first inequality in \eqref{QuadRegMCD_LinearRateOfConvergence} one obtains that the second
inequality in \eqref{QuadRegMCD_LinearRateOfConvergence} is valid.	 
\end{proof}

Let us show that if the QR-MCD converges to a point satisfying a certain \textit{first} order sufficient optimality
condition, then it converges \textit{quadratically}. To this end, let us recall the definition of \textit{Haar
point} \cite{Polak} (or \textit{Chebyshev point} in the terminology of the paper \cite{DaugavenMalozemov}). 

Let $\mathcal{H} = \mathbb{R}^d$. Fix arbitrary $x^* \in \mathbb{R}^d$ and 
$\lambda = (j_1, \ldots, j_s) \in J(x^*)$. Suppose that $x^*$ is an inf-stationary point 
of the function $f_{\lambda}$, i.e. $0 \in \underline{d} f_{\lambda}(x^*)$ (note that $f_{\lambda}$ is
obviously hypodifferentiable). Then there exist $\alpha_{\zeta} \ge 0$, $\zeta = (i_1, \ldots, i_m) \in I := I_1 \times
\ldots \times I_m$, such that $\alpha_{\zeta} = 0$ for all $\zeta \notin I(x^*)$, and
$$
  \sum_{\zeta \in I} \alpha_{\zeta} \Big( \sum_{k = 1}^m \nabla g_{k i_k} (x^*) \Big)
  + \sum_{l = 1}^s \nabla u_{l j_l}(x^*) = 0, \quad
  \sum_{\zeta \in I} \alpha_{\zeta} = 1.
$$
Denote the set of all such $\alpha = \{ \alpha_{\zeta} \}_{\zeta \in I}$ by $\alpha_{\lambda}(x^*)$. The elements of 
the set $\alpha_{\lambda}(x^*)$ are sometimes called \textit{Danskin-Demyanov multipliers} \cite{Polak}.

\begin{definition}
The point $x^*$ is called a \textit{Haar point} of the function $f_{\lambda}$, if
\begin{enumerate}
\item{$0 \in \underline{d} f_{\lambda}(x^*)$,
\label{Def_Stationarity}}

\item{$|I(x^*)| = |I_{1, 0}(x^*)| \times \ldots \times |I_{m, 0}(x^*)| = d + 1$, where $|M|$ is the cardinality
of a set $M$,
\label{Def_EnoughVertices}}

\item{the vectors $\nabla g_{1 i_1}(x^*) + \ldots + \nabla g_{m i_m}(x^*)$, $(i_1, \ldots, i_m) \in I(x^*)$, are
affinely independent,
\label{Def_CompleteAlternance}}

\item{for any $\alpha \in \alpha_{\lambda}(x^*)$ one has $\alpha_{\zeta} > 0$ for all $\zeta \in I(x^*)$.
\label{Def_SuffMinCond}}
\end{enumerate}
\end{definition}

\begin{remark}
Note that if assumptions \ref{Def_EnoughVertices} and \ref{Def_CompleteAlternance} from the definition
above are satisfied, then assumptions \ref{Def_Stationarity} and \ref{Def_SuffMinCond} hold true iff 
$0 \in \interior \underline{\partial} f_{\lambda}(x^*)$. Furthermore, it is easy to see that if $x^*$ is a Haar point of
$f_{\lambda}$, then the set $\alpha_{\lambda}(x^*)$ consists of only one element.
\end{remark}

Let $x^*$ be a Haar point of $f_{\lambda}$. Then by \cite{Polak}, Theorem~2.4.22 (see also~\cite{DaugavenMalozemov}) for
any starting point $x_0$ sufficiently close to $x^*$ the PPP algorithm for the function $f_{\lambda}$ coincides with the
Newton method for solving a certain system of nonlinear equations, and converges quadratically to $x^*$. Thus, there
exists a neighbourhood $U$ of $x^*$ such that for any starting point $x_0 \in U$ the sequence $\{ x_n \}$ generated by 
the PPP algorithm for the function $f_{\lambda}$ stays in $U$ and converges quadratically to $x^*$.

\begin{theorem}
Let $\mathcal{H} = \mathbb{R}^d$, $\alpha_* \ge 1$, $\nu_n \equiv + \infty$, and $\lim_{n \to \infty} \mu_n > 0$. Let
also the gradients of the functions $g_{ki}$ and $u_{lj}$ be locally Lipschitz continuous. Suppose, finally, that a
sequence $\{ x_n \}$ generated by the QR-MCD for the function $f$ converges to a point $x^*$ such that for any 
$\lambda \in J(x^*)$ the point $x^*$ is a Haar point of the function $f_{\lambda}$. Then $x^*$ is a point of strict
local minimum of the function $f$, and $x_n \to x^*$ as $n \to \infty$ quadratically.
\end{theorem}

\begin{proof}
Note that by Theorem~\ref{Th_QuadReg_GlobalConvergence} the point $x^*$ is an inf-stationary point of the function $f$,
which implies that $0 \in \underline{d} f_{\lambda}(x^*)$ for all $\lambda \in J(x^*)$. Applying \cite{Polak},
Lemma~2.4.19 (see also inequality (2.4.24g)), and the fact that $x^*$ is a Haar point of every function $f_{\lambda}$,
$\lambda \in J(x^*)$, one obtains that there exist a neighbourhood $U_0$ of $x^*$ and $\varkappa > 0$ such that
\begin{equation} \label{MaxFunc_FirstOrderGrowth}
  f_{\lambda}(x) - f_{\lambda}(x^*) \ge \varkappa \| x - x^* \| 
  \quad \forall x \in U_0 \quad \forall \lambda \in J(x^*).
\end{equation}
Hence, as it is easily seen, there exists a neighbourhood $U \subseteq U_0$ of $x^*$ such that 
$f(x) - f(x^*) \ge \varkappa \| x - x^* \|$ for all $x \in U$, i.e. $x^*$ is a point of strict local minimum of the
function $f$.

From the continuity of the functions $u_{lj}$ it follows that there exists a neighbourhood $V \subseteq U$ of $x^*$ such
that $J(x) \subseteq J(x^*)$ for all $x \in V$. Taking into account the fact that $x_n \to x^*$ as $n \to \infty$,
one obtains that there exists $n_0 \in \mathbb{N}$ such that $x_n \in V$ for all $n \ge n_0$.

Fix arbitrary $\sigma, \gamma \in (0, 1)$, $n \ge n_0$ and $\lambda = (j_1, \ldots, j_s) \in J(x_n)$. Define
$z_n = ( 0, \sum_{l = 1}^s \nabla u_{l j_l}(x_n) )$, and let $\gamma_n$ be defined as in the proof of
Theorem~\ref{Th_LinearConvergenceOfQuadRegMCD}. Then $\gamma_n$ is the step size in the PPP algorithm for the function
$f_{\lambda}$. Note that one can choose $n_0$ so large that $x_n$ belongs to the region of quadratic convergence of 
the PPP algorithm for all function $f_{\lambda}$, $\lambda \in J(x^*)$, and a sequence generated
by this algorithms for any function $f_{\lambda}$, $\lambda \in J(x^*)$ with the starting point $x_n$ stays in $V$.
Therefore there exists $Q > 0$ (independent of $n \ge n_0$ and $\lambda \in J(x^*)$) such that 
$$
  \| x_n + \gamma_n h_n(z_n) - x^* \| \le Q \| x_n - x^* \|^2, \quad x_n + \gamma_n h_n(z_n) \in V. 
$$
By Corollary~\ref{Corollary_LipschitzContinuity} the function $f$ is Lipschitz continuous on any bounded set. Hence and
from \eqref{MaxFunc_FirstOrderGrowth} it follows that there exists $L > 0$ such that
\begin{multline*}
  f(x_n + \gamma_n h_n(z_n)) - f(x^*) \le L \| x_n + \gamma_n h_n(z_n) - x^* \| \\
  \le L Q \| x_n - x^* \|^2 \le \frac{L Q}{\varkappa^2} \big[ f(x_n) - f(x^*) \big]^2
\end{multline*}
Recall that $\alpha_* \ge 1$. Hence $f(x_{n + 1}) \le f(x_n + \alpha_n(z_n) h_n(z_n)) \le f(x_n + \gamma_n h_n(z_n))$.
Consequently, one has
$$
  \varkappa \| x_{n + 1} - x^* \| \le f(x_{n + 1}) - f(x^*) 
  \le \frac{L Q}{\varkappa^2} \big[ f(x_n) - f(x^*) \big]^2 \le \frac{L^3 Q}{\varkappa^2} \| x_n - x^* \|^2,
$$
i.e. $x_n \to x^*$ as $n \to \infty$ quadratically.	 
\end{proof}

\section{Conclusion}

In this paper we studied two methods for minimizing a codifferentiable function defined on a Hilbert space. Namely, we
proposed and analysed a generalization of the method of codifferential descent and a quadratic regularization
of the method of codifferential descent. In order to study a convergence of these methods, we introduced a class
of uniformly codifferentiable functions, and derived some calculus rules that allow one to verify whether a given
nonsmooth function is uniformly codifferentiable, and compute the corresponding codifferential mapping. Under some
natural assumptions we proved the global convergence of the methods, as well as the convergence of the inf-stationarity
measure $\omega(x_n, \nu_n)$ to zero. It should be noted that the quadratic regularization of the MCD (as well as the
primal regularization of this method) is the first method for minimizing a codifferentiable function over a convex set.
In the end of the paper we estimated the rate of convergence of the MCD for a class of nonsmooth nonconvex functions.

\section{Acknowledgements}

The author is sincerely grateful to his colleagues T. Angelov, A. Fominyh and G.Sh. Tamasyan for numerous useful
discussions on the method of codifferential descent and its applications that played a significant role in 
the prepration of this paper. Also, the author wishes to express his thanks to professor V.N. Malozemov for pointing out
the fact that the PPP algorithm converges quadratically to a Chebyshev (Haar) point.

\bibliographystyle{abbrv}  
\bibliography{Codiff_Descent_bibl}

\begin{thebibliography}{10}

\bibitem{Andramonov}
M.~Y. Andramonov.
\newblock The method of confidence neighbourhoods for minimization of
  codifferentiable functions.
\newblock {\em Russ. Math.}, 48:1--7, 2004.

\bibitem{Bagirov2002}
A.~Bagirov.
\newblock A method for minimization of quasidifferentiable functions.
\newblock {\em Optim. Methods Softw.}, 17:31--60, 2002.

\bibitem{BagirovKarmitsaMakela_book}
A.~Bagirov, N.~Karmitsa, and M.~M. M\"{a}kel\"{a}.
\newblock {\em Introduction to Nonsmooth Optimization}.
\newblock Springer International Publishing, Cham, 2014.

\bibitem{BagGanUgonTor}
A.~M. Bagirov, A.~N. Ganjehlou, J.~Ugon, and A.~H. Tor.
\newblock Truncated codifferential method for nonsmooth convex optimization.
\newblock {\em Pac. J. Optim.}, 6:483--496, 2010.

\bibitem{BarigovKarasozenSezer}
A.~M. Bagirov, B.~Karas\"{o}zen, and M.~Sezer.
\newblock Discrete gradient method: derivative-free method for nonsmooth
  optimization.
\newblock {\em J. Optim. Theory Appl.}, 137:317--334, 2008.

\bibitem{BagirovRubinovZhang}
A.~M. Bagirov, A.~M. Rubinov, and J.~Zhang.
\newblock A multidimensional descent method for global optimization.
\newblock {\em Optim.}, 58:611--625, 2009.

\bibitem{BagUgon}
A.~M. Bagirov and J.~Ugon.
\newblock Codifferential method for minimizing nonsmooth dc functions.
\newblock {\em J. Glob. Optim.}, 50:3--22, 2011.

\bibitem{Barirog_Survey}
A.~M. Bagriov.
\newblock Numerical methods for minimizing quasidifferentiable functions: a
  survey and comparison.
\newblock In V.~Demyanov and A.~Rubinov, editors, {\em Quasidifferentiability
  and Related Topics}, pages 33--71. Kluwer Academic Publishers, Dordrecht,
  2000.

\bibitem{BurkeLewisOverton}
J.~V. Burke, A.~S. Lewis, and M.~L. Overton.
\newblock A robust gradient sampling algorithm for nonsmooth, nonconvex
  optimization.
\newblock {\em SIAM J. Optim.}, 15:751--779, 2005.

\bibitem{Clarke}
F.~H. Clarke.
\newblock {\em Optimization and Nonsmooth Analysis}.
\newblock Wiley--Interscience, New York, 1983.

\bibitem{ConnScheinbergVicent_book}
A.~R. Conn, K.~Scheinberg, and L.~N. Vicente.
\newblock {\em Introduction to Derivative-Free Optimization}.
\newblock SIAM, Philadelphia, 2009.

\bibitem{CurtisQue2013}
F.~E. Curtis and X.~Que.
\newblock An adaptive gradient sampling algorithm for non-smooth optimization.
\newblock {\em Optim. Methods Softw.}, 28:1302--1324, 2013.

\bibitem{DaugavenMalozemov}
V.~A. Daugavet and V.~N. Malozemov.
\newblock Quadratic rate of convergence of a linearization method for solving
  discrete minimax problems.
\newblock {\em USSR Comput. Math. Math. Phys.}, 21:19--28, 1981.

\bibitem{Demyanov_InCollection_1988}
V.~F. Demyanov.
\newblock Continuous generalized gradients for nonsmooth functions.
\newblock In A.~Kurzhanski, K.~Neumann, and D.~Pallaschke, editors, {\em
  Optimization, Parallel Processing and Applications}, pages 24--27. Springer,
  Berlin, Heidelberg, 1988.

\bibitem{Demyanov1988}
V.~F. Demyanov.
\newblock On codifferentiable functions.
\newblock {\em Vestn. Leningr. Univ., Math.}, 2:22--26, 1988.

\bibitem{Demyanov1989}
V.~F. Demyanov.
\newblock Smoothness of nonsmooth functions.
\newblock In F.~H. Clarke, V.~F. Demyanov, and F.~Giannesssi, editors, {\em
  Nonsmooth Optimization and Related Topics}, pages 79--88. Springer, Boston,
  1989.

\bibitem{DemBagRub}
V.~F. Demyanov, A.~M. Bagirov, and A.~M. Rubinov.
\newblock A method of truncated codifferential with applications to some
  problems of cluster analysis.
\newblock {\em J. Glob. Optim.}, 23:63--80, 2002.

\bibitem{DemyanovMalozemov}
V.~F. Demyanov and V.~N. Malozemov.
\newblock {\em Introduction to Minimax}.
\newblock Wiley and Sons, New York, 1990.

\bibitem{DemRub_book}
V.~F. Demyanov and A.~M. Rubinov.
\newblock {\em Constructive Nonsmooth Analysis}.
\newblock Peter Lang, Frankfurt am Main, 1995.

\bibitem{Quasidifferentiability_book}
V.~F. Demyanov and A.~M. Rubinov, editors.
\newblock {\em Quasidifferentiability and Related Topics}.
\newblock Kluwer Academic Publishers, Dordrecht, 2000.

\bibitem{QuasidiffMechanics}
V.~F. Demyanov, G.~E. Stavroulakis, L.~N. Polyakova, and P.~D. Panagiotopoulos.
\newblock {\em Quasidifferentiability and Nonsmooth Modelling in Mechnics,
  Engineering and Economics}.
\newblock Kluwer Academic Publishers, Dordrecht, 1996.

\bibitem{DemyanovTamasyan2011}
V.~F. Demyanov and G.~S. Tamasyan.
\newblock Exact penalty functions in isoperimetric problems.
\newblock {\em Optim.}, 60:153--177, 2011.

\bibitem{DemyanovTamasyan2014}
V.~F. Demyanov and G.~S. Tamasyan.
\newblock Direct methods in the parametric moving boundary variational problem.
\newblock {\em Numer. Funct. Anal. Optim.}, 35:932--961, 2014.

\bibitem{ExactBarrierFunc}
G.~Di~Pillo and F.~Facchinei.
\newblock Exact barrier function methods for lipschitz programs.
\newblock {\em Appl. Math. Optim.}, 32:1--31, 1995.

\bibitem{Dolgopolik}
M.~V. Dolgopolik.
\newblock Codifferential calculus in normed spaces.
\newblock {\em J. Math. Sci.}, 173:441--462, 2011.

\bibitem{Dolgopolik_NonhomConvApprox}
M.~V. Dolgopolik.
\newblock Inhomogeneous convex approximations of nonsmooth functions.
\newblock {\em Russ. Math.}, 56:28--42, 2012.

\bibitem{Dolgopolik_CalcVar}
M.~V. Dolgopolik.
\newblock Nonsmooth problems of calculus of variations via codifferentiation.
\newblock {\em ESAIM: Control, Optim. Calc. Var.}, 20:1153--1180, 2014.

\bibitem{Dolgopolik_AbstractConv}
M.~V. Dolgopolik.
\newblock Abstract convex approximations of nonsmooth functions.
\newblock {\em Optim.}, 64:1439--1469, 2015.

\bibitem{Edwards}
R.~E. Edwards.
\newblock {\em Functional Analysis: Theory and Applications}.
\newblock Dover Publications, New York, 1995.

\bibitem{Fominyh}
A.~V. Fominyh, V.~V. Karelin, and L.~N. Polyakova.
\newblock Application of the hypodifferential descent method to the problem of
  constructing an optimal control.
\newblock {\em Optim. Lett.}, 2017.
\newblock doi: 10.1007/s11590-017-1222-x.

\bibitem{FuduliGaudioso}
A.~Fuduli, M.~Gaudioso, and E.~A. Nurminski.
\newblock A splitting bundle approach for non-smooth non-convex minimization.
\newblock {\em Optim.}, 64:1131--1151, 2015.

\bibitem{HaaralaMiettinenMakela}
N.~Haarala, K.~Miettinen, and M.~M\"{a}kel\"{a}.
\newblock Globally convergent limited memory bundle method for large-scale
  nonsmooth optimization.
\newblock {\em Math. Program.}, 109:181--205, 2007.

\bibitem{HareNutini}
W.~Hare and J.~Nutini.
\newblock A derivative-free approximate gradient sampling algorithm for finite
  minimax problems.
\newblock {\em Comput. Optim. Appl.}, 56:1--38, 2013.

\bibitem{HareSagatizabal}
W.~Hare and C.~Sagastiz\'{a}bal.
\newblock A redistributed proximal bundle method for nonconvex optimization.
\newblock {\em SIAM J. Optim.}, 20:2442--2473, 2010.

\bibitem{HiriartUrrutyLemarechal}
J.-B. Hiriart-{U}rruty and C.~Lemar\'echal.
\newblock {\em Convex Analysis and Minimization Algorithms I}.
\newblock Springer-Verlag, Berlin, Heidelberg, 1993.

\bibitem{IoffeTihomirov}
A.~D. Ioffe and V.~M. Tihomirov.
\newblock {\em Theory of Extremal Problems.}
\newblock North-Holland Publishing Company, Amsterdam etc., 1979.

\bibitem{KarmitsaBagriovMakela2012}
N.~Karmitsa, A.~Bagirov, and M.~M. M\"{a}kel\"{a}.
\newblock Comparing different nonsmooth minimization methods and software.
\newblock {\em Optim. Methods Softw.}, 27:131--153, 2012.

\bibitem{KeskarWachter}
N.~Keskar and A.~W\"{a}chter.
\newblock A limited-memory quasi-newton algorithm for bound-constrained
  non-smooth optimization.
\newblock {\em Optim. Methods Softw.}, 2017.
\newblock doi: 10.1080/10556788.2017.1378652.

\bibitem{Kiwiel2010}
K.~C. Kiwiel.
\newblock A nonderivative version of the gradient sampling algorithm for
  nonsmooth nonconvex optimization.
\newblock {\em SIAM J. Optim.}, 20:1983--1994, 2010.

\bibitem{Kuntz}
L.~Kuntz.
\newblock A charaterization of continuously codifferentiable function and some
  consequences.
\newblock {\em Optim.}, 22:539--547, 1991.

\bibitem{LewisOverton2013}
A.~S. Lewis and M.~L. Overton.
\newblock Nonsmooth optimization via quasi-newton methods.
\newblock {\em Math. Program.}, 141:135--163, 2013.

\bibitem{Luderer}
B.~Luderer.
\newblock Does the special choice of quasidifferentials influence necessary
  minimum conditions?
\newblock In W.~Oettli and D.~Pallaschke, editors, {\em Advances in
  Optimization}, pages 256--266. Springer-Verlag, Berlin, Heidelberg, 1992.

\bibitem{LudererRosigerWurker}
B.~Luderer, R.~R\"{o}siger, and U.~W\"{u}rker.
\newblock On necessary minimum conditions in quasidifferential calculus:
  independence of the specific choice of quasidifferentials.
\newblock {\em Optim.}, 22:643--660, 1991.

\bibitem{LudererWeigelt2003}
B.~Luderer and J.~Weigelt.
\newblock A solution method for a special class of nondifferentiable
  unconstrained optimization problems.
\newblock {\em Comput. Optim. Appl.}, 24:83--93, 2003.

\bibitem{Makela2002}
M.~M\"{a}kel\"{a}.
\newblock Survey of bundle method for nonsmooth optimization.
\newblock {\em Optim. Methods Softw.}, 17:1--29, 2002.

\bibitem{Morukhovich}
B.~S. Mordukhovich.
\newblock {\em Variational Analysis and Generalized Differentiation I. Basic
  Theory}.
\newblock Springer-Verlag, Berlin, Heidelberg, 2006.

\bibitem{PalRecUrb}
D.~Pallaschke, P.~Recht, and R.~Urba\'nski.
\newblock On locally-lipschitz quasi-differentiable functions in banach spaces.
\newblock {\em Optim.}, 17:287--295, 1986.

\bibitem{PallaschkeUrbnski}
D.~Pallaschke and R.~Urba\'{n}ski.
\newblock {\em Pairs of Compact Convex Sets: Fractional Arithmetic with Convex
  Sets}.
\newblock Kluwer Academic Publishers, Dordrecht, 2002.

\bibitem{Penot}
J.-P. Penot.
\newblock On the convergence of descent algorithms.
\newblock {\em Comput. Optim. Appl.}, 23:279--284, 2002.

\bibitem{Penot_book}
J.-P. Penot.
\newblock {\em Calculus Without Derivatives}.
\newblock Springer Science+Business Media, New York, 2013.

\bibitem{PironneauPolak}
O.~Pironneau and E.~Polak.
\newblock On the rate of convergence of certain methods of centers.
\newblock {\em Math. Program.}, 2:230--258, 1972.

\bibitem{Polak}
E.~Polak.
\newblock {\em Optimization. Algorithms and Consistent Approximations}.
\newblock Springer-Verlag, New York, 1997.

\bibitem{Pschenichny}
B.~N. Pschenichny and Y.~M. Danilin.
\newblock {\em Numerical Methods in Extremal Problems. English ed.}
\newblock Mir Publishers, Moscow, 1978.

\bibitem{TorBagKar}
A.~H. Tor, A.~Bagirov, and B.~Karas\"ozen.
\newblock Aggregate codifferential method for nonsmooth dc optimization.
\newblock {\em J. Comput. Appl. Math.}, 259:851--867, 2014.

\bibitem{TorKarBag}
A.~H. Tor, B.~Karas\"ozen, and A.~Bagirov.
\newblock Truncated codifferential method for linearly constrained nonsmooth
  optimization.
\newblock In {\em the ISI Proc. of 24th Mini EURO Conference ``Continuous
  Optimization and Information-Based Technologies in the Financial Sector'',
  Izmir, Turkey}, pages 87--93, 2010.

\bibitem{Uderzo}
A.~Uderzo.
\newblock Fr\'echet quasidifferential calculus with applications to metric
  regularity of continuous maps.
\newblock {\em Optim.}, 54:469--493, 2005.

\bibitem{Zaffaroni}
A.~Zaffaroni.
\newblock Continuous approximations, codifferentiable functions and
  minimization methods.
\newblock In V.~F. Demyanov and A.~M. Rubinov, editors, {\em
  Quasidifferentiability and related Topics}, pages 361--391. Kluwer Academic
  Publishers, Dordrecht, 2000.

\end{thebibliography}

\end{document}